\documentclass[10pt]{article}

\usepackage[margin=1in]{geometry}

\usepackage{amsmath,amsthm,amssymb,hyperref,bm,amsbsy,graphicx,epstopdf}

\usepackage{bbm}

\usepackage{threeparttable}
\usepackage{booktabs}

\usepackage{graphicx,comment,float}

\usepackage{caption}
\usepackage{subcaption}

\numberwithin{equation}{section}

\usepackage[ruled,linesnumbered]{algorithm2e}
\RestyleAlgo{ruled}
\SetKwComment{Comment}{/* }{ */}
\let\oldnl\nl
\newcommand{\nonl}{\renewcommand{\nl}{\let\nl\oldnl}}

\usepackage{color,soul}

\usepackage{array}

\usepackage{soul}
\usepackage[normalem]{ulem}

\hypersetup{
    colorlinks=true, 
    linkcolor=blue,
    citecolor=red,
    filecolor=magenta,
    urlcolor=blue,
}

\newtheorem{thm}{{Theorem}}

\newtheorem{prop}{{Proposition}}
\newtheorem{lem}{{Lemma}}

\newtheorem{rem}{{Remark}}

\let\oldproofname=\proofname
\renewcommand{\proofname}{\rm\bf{\oldproofname}}

\newcommand{\R}{\mathbb{R}}

\newcommand*\Bell{\ensuremath{\boldsymbol\ell}}

\DeclareMathOperator{\arctanh}{arctanh}

\title{On Sparse Grid Interpolation for American Option Pricing with Multiple Underlying Assets}
\author{Jiefei Yang \thanks{Department of Mathematics, The University of Hong Kong, Pokfulam Road, Hong Kong. Email: {\tt{jiefeiy@connect.hku.hk}} JY acknowledges support from the University of Hong Kong via the HKU Presidential PhD Scholar Programme (HKU-PS). The authors would like to thank the Isaac Newton Institute for Mathematical Sciences for support and hospitality during the programme Uncertainty Quantification and Stochastic Modelling of Materials when work on this paper was undertaken. This work was supported by EPSRC Grant Number EP/R014604/1.} \quad and \quad Guanglian Li\thanks{Department of Mathematics, The University of Hong Kong, Pokfulam Road, Hong Kong. Email: {\tt{lotusli@maths.hku.hk}} GL acknowledges the support from GRF (project number: 17317122) and Early Career Scheme (Project number: 27301921), RGC, Hong Kong. The authors thanks Prof. Michael Griebel (University of Bonn, Germany) for providing valuable literatures on option pricing and sparse grids.}}
\date{September 15, 2023}

\begin{document}
\maketitle
\begin{abstract}
In this work, we develop a novel efficient quadrature and sparse grid based polynomial interpolation method to price American options with multiple underlying assets. The approach is based on first formulating the pricing of American options using dynamic programming, and then employing static sparse grids to interpolate the continuation value function at each time step. To achieve high efficiency, we first transform the domain from $\mathbb{R}^d$ to $(-1,1)^d$ via a scaled tanh map, and then remove the boundary singularity of the resulting multivariate function over $(-1,1)^d$ by a bubble function and simultaneously, to significantly reduce the number of interpolation points. We rigorously establish that with a proper choice of the bubble function, the resulting function has bounded mixed derivatives up to a certain order, which provides theoretical underpinnings for the use of sparse grids. Numerical experiments for American arithmetic and geometric basket put options with the number of underlying assets up to 16 are presented to validate the effectiveness of the approach.

\end{abstract}

\section{Introduction}
This paper is concerned with American option pricing with payoffs affected by many underlying instruments, which can be assets such as stocks, bonds, currencies, commodities, and indices (e.g., S\&P 500, NASDAQ 100) \cite[p. 365]{zhang1997exotic}. Classical examples include pricing American max-call and basket options \cite{andersen2004primal, haugh2004pricing, hu2020pricing, kovalov2007pricing, nielsen2008penalty, reisinger2007efficient, scheidegger2021pricing}. In practice, American options can be exercised only at discrete dates. The options that can be exercised only at finite discrete dates are called Bermudan options, named after the geographical feature of Bermudan Islands (i.e., being located between America and Europe while much closer to the American seashore) \cite{guyon2013nonlinear}.  Merton \cite{merton1976option} showed that pricing an American call option on a single asset with no dividend is equivalent to pricing a European option, and obtained an explicit solution to pricing perpetual American put. However, except these two cases no closed-form solution is known to price American options even in the simplest Black Scholes model. Thus it is imperative to develop efficient numerical methods for American option pricing.

Various numerical methods have been proposed for pricing and hedging in the past five decades, using different formulations
of American option pricing, e.g., an optimal stopping problem, a variational inequality, or a free boundary problem
\cite{bensoussan2011applications,jaillet1990variational, McKean1965,van1974optimal}. A finite difference method (FDM) was
proposed to price American options based on variational inequalities \cite{brennan1977valuation}, with its convergence
proved in \cite{jaillet1990variational} by showing that the $\mathcal{C}^1$ regularity of the value function with respect
to the underlying price. The binomial options pricing model (BOPM) based upon optimal time stopping was developed
in \cite{cox1979option}, and its convergence was shown in \cite{amin1994convergence}. When the number $d $ of underlying
assets is smaller than four, one can extend one-dimensional pricing methods using tensor product or additional treatment
to price multi-asset options. For example, Cox-Ross-Rubinstein (CRR) binomial tree model
can be extended to the multinomial option pricing model to price American options with two underlying assets
\cite{boyle1988lattice}. However, dynamic programming (cf. \eqref{eq: DP} below) or variational inequalities
\cite{bensoussan1984theory} are predominant when $d$ is greater than three.
Many numerical schemes have been developed based on the variational inequalities, e.g.,
FDMs \cite{achdou2005computational,brennan1977valuation, forsyth2002quadratic} and finite element methods (FEMs)
\cite{kovalov2007pricing}. However, only the first order convergence rate can be achieved, since the value function has only $\mathcal{C}^1$ regularity with respect to the underlying price (i.e., smooth pasting condition \cite{peskir2006optimal}).

There are mainly two lines of research on American option pricing based on dynamic programming, i.e., simulation
based methods \cite{andersen2004primal, longstaff2001valuing,tsitsiklis2001regression} and quadrature and
interpolation (Q\&I) based methods \cite{quecke2007efficient,sullivan2000valuing}. Simulation based methods are
fast and easy to implement, but their accuracy is hard to justify. One representative method is the least square
Monte Carlo method \cite{longstaff2001valuing}, which employs least square regression and Monte Carlo method to
approximate conditional expectations, cf. \eqref{eq: DP}. Q\&I based methods employ quadrature to approximate
conditional expectations and interpolation to construct function approximators. One can use Gaussian quadrature
or adaptive quadrature to approximate conditional expectations in \eqref{eq: DP} and Chebyshev polynomial
interpolation, spline interpolation or radial basis functions to reconstruct the continuation value or the value
function \cite{quecke2007efficient,sullivan2000valuing}. In \cite{glau2019new}, a dynamic Chebyshev method via
polynomial interpolation of the value functions was developed, allowing the generalized moments evaluation in the offline stage to reduce computational complexity. Although Q\&I based methods are efficient in low-dimensional settings, the extensions to high-dimensional cases are highly nontrivial, and there are several outstanding challenges, e.g., curse of dimensionality, unboundedness of the domain and the absence of natural boundary conditions. For example, if the domain is truncated and artificial boundary conditions are imposed, then one is actually pricing an American barrier option with rebate instead of the American option itself. Accurate boundary conditions can be obtained by pricing a $(d-1)$-dimensional problem \cite{nielsen2008penalty}, which, however, is still computationally challenging.

The sparse grids technique has been widely used in option pricing with multiple underlying assets, due to its capability to approximate high-dimensional functions with bounded mixed derivatives, for which the computational complexity for a given accuracy does not grow exponentially with respect to $d$ \cite{bungartz2004sparse}. It has been applied to price multi-asset European and path-dependent Asian options, by formulating them as a high-dimensional integration problem \cite{bayer2021numerical,bungartz2003multivariate, holtz2010sparse}, and also combined with the FDMs \cite{reisinger2007efficient} and FEMs \cite{bungartz2012option} to price options with $d\leq 5$. In the context of Q\&I methods, adaptive sparse grids interpolation with local linear hierarchical basis have been used to approximate value functions \cite{scheidegger2021pricing}.

In this work, we propose a novel numerical approach to price American options under multiple underlying assets, which is summarized in Algorithm \ref{alg: SGSG}. It crucially draws on the $\mathcal{C}^\infty$ regularity of the continuation value function, cf. \eqref{eq:continuation}, and uses sparse grid Chebyshev polynomial interpolation to alleviate the curse of dimensionality. This is achieved in several crucial steps. First, we transform the unbounded domain $\mathbb{R}^d$ into a bounded one, which eliminates the need of imposing artificial boundary conditions. Second, to further improve the computational efficiency, using a suitable bubble function, we obtain a function that can be continuously extended to the boundary with vanishing boundary values and with bounded mixed derivatives up to certain orders, which is rigorously justified in Theorem \ref{thm: bounded-mixed}. This construction enables the use of the standard sparse grid technique with much fewer sparse grids (without adaptivity), and moreover, the interpolation functions fulfill the requisite regularity conditions, thereby admitting theoretical convergence guarantees. The distinct features of the proposed method include using static sparse grids at all time steps and allowing deriving the value function on the whole domain $\mathbb{R}^d$, and thus can also be used to estimate important parameters, e.g., hedge ratio.

Extensive numerical experiments demonstrate that Algorithm \ref{alg: SGSG} can
break the curse of dimensionality in the sense that high accuracy is achieved with the associated computational complexity being almost independent of the dimension. Furthermore, our experiments can significantly enrich feasible dimensionality, e.g., pricing American arithmetic basket put options up to $d=12$ (versus $d=6$ from the literature), and pricing American geometric basket put options up to $d=16$. In both cases, we consistently observe high accuracy of the approximation, with almost no influence from the dimension. Further experimental evaluations show that the method is robust with respect to the choice of various algorithmic parameters and the specifics of the quadrature scheme. In sum, our proposed algorithm significantly improves over the state-of-the-art pricing schemes for American options with multiple underlying assets and in the meanwhile has rigorous theoretical guarantee.

The remainder of our paper is structured as follows. In Section \ref{sec:prelim}, we derive the continuation value function $f_k(\mathbf{x})$ for an American basket put option with underlying assets following the correlated geometric Brownian motion under the risk-neutral probability. In Section \ref{sec:method}, we develop the novel method. In Section \ref{sec:analysis}, we establish the smoothness of the interpolation function in the space of functions with bounded mixed derivatives. In Section \ref{sec:numerical}, we present extensive numerical tests for American basket options with up to 16 underlying assets. Finally, we conclude with future work in Section \ref{sec:conclusion}.

\section{The mathematical model}\label{sec:prelim}

Martinagle pricing theory gives the fair price of an American option as the solution to the optimal stopping problem in the risk-neutral probability space $(\Omega, \mathcal{F}, (\mathcal{F}_t)_{0\le t\le T}, \mathbb{Q})$:
\begin{equation} \label{eq: OSP}
    V(t) = \sup_{\tau_t \in [t, T]} \mathbb{E}[e^{-r(\tau_t - t)} g(\mathbf{S}(\tau_t)) | \mathcal{F}_t],
\end{equation}
where $\tau_t$ is a $\mathcal{F}_t$-stopping time, $T$ is the expiration date, $(\mathbf{S}(t))_{0\le t\le T}$, is a collection of $d$-dimensional price processes, and $g(\cdot)$ is the payoff function depending on the type of the option. The payoffs of put and call options take the following form $g(\mathbf{S}) = \max( \kappa - \Psi(\mathbf{S}) , 0)$ and $g(\mathbf{S}) = \max( \Psi(\mathbf{S}) - \kappa , 0)$,
respectively, where $\Psi: \mathbb{R}_{+}^d\to\mathbb{R}_{+}$, and $\kappa$ is the strike price.

Now we describe a detailed mathematical model for pricing an American put option on $d$ underlying assets with a strike price
$\kappa$ and a maturity date  $T$, whose numerical approximation is the main objective of this work. One classical
high-dimensional example is pricing American basket options. Let $g:\mathbb{R}_{+}^d\to\mathbb{R}_{+} $ be its payoff function,
and assume that the prices of the underlying assets $\mathbf{S}(t) = [S^1(t), \dots, S^d(t)]^\top$ follow the correlated geometric Brownian motions
\begin{equation}\label{eq:model}
    \mathrm{d}S^i(t) = (r - \delta_i) S^i(t)\,\mathrm{d}t + \sigma_i S^i(t)\,\mathrm{d}\tilde{W}^i(t) \quad \text{ with } S^i(0) = S^i_0, \quad i = 1,2,\dots,d,
\end{equation}
where $\tilde{W}^i(t)$ are correlated $\mathbb{Q}$-Brownian motions with  $\mathbb{E}[\mathrm{d}\tilde{W}^i(t) \mathrm{d}\tilde{W}^j(t)] = \rho_{ij} \,\mathrm{d}t$,
$\rho_{ii} = 1$ for $i,j = 1,2,\dots, d$, and $r$, $\delta_i$ and $\sigma_i$ are the riskless interest rate, dividend yields,
and volatility parameters, respectively.
The payoffs of an arithmetic and a geometric basket put are respectively given by
\begin{equation*}
    g(\mathbf{S}) = \max\left(\kappa - \frac{1}{d}\sum_{i=1}^d S^i, 0\right)\quad \mbox{and}\quad     g(\mathbf{S}) = \max\left(\kappa - \left(\prod_{i=1}^d S^i\right)^{1/d}, 0\right).
\end{equation*}

In practice, the $\mathcal{F}_0$-stopping time $\tau_0$ is assumed to be taken in a set of discrete time steps, $\mathcal{T}_{0,T} := \{t_k = k\Delta t: k=0,1,\ldots,K\}$, with $\Delta t=  T/K $. This leads to the pricing of a $K$-times exercisable Bermudan option that satisfies the following dynamic programming problem
\begin{equation} \label{eq: DP}
    \begin{split}
        & \mathcal{V}_K(\mathbf{s}) = g(\mathbf{s}), \\
        & \mathcal{V}_k(\mathbf{s}) = \max\big( g(\mathbf{s}), \mathbb{E}[e^{-r\Delta t} \mathcal{V}_{k+1}(\mathbf{S}_{k+1}) | \mathbf{S}_k = \mathbf{s} ] \big)\quad\text{ for } k = K-1:-1:0,
    \end{split}
\end{equation}
where $\mathcal{V}_k$ is called the value function at time $t_k$ and $\mathbf{S}_k := \mathbf{S}(t_k)$ is the discretized price processes. Throughout, for a stochastic process $\mathbf{X}(t)$, we write $\mathbf{X}_k := \mathbf{X}(t_k)$, $k=0,1,\dots, K$. Note that by the Markov property of It\^{o} process, we can substitute the conditional expectation conditioning on $\mathcal{F}_{t_k}$ to $\mathbf{S}_k$. The conditional expectation as a function of prices is called the continuation value function, i.e.,
\begin{equation}\label{eq:continuation}
    C_k(\mathbf{s}) = \mathbb{E}[e^{-r\Delta t} \mathcal{V}_{k+1}(\mathbf{S}(t_{k+1})) | \mathbf{S}(t_k) = \mathbf{s} ]\quad \text{ for } k = 0,1,\dots, K-1.
\end{equation}
Below we recast problem \eqref{eq: DP} in terms of the continuation value function
\begin{equation} \label{eq: DPC}
    \begin{aligned}
         C_{K-1}(\mathbf{s}) &= \mathbb{E}\left[e^{-r\Delta t} g(\mathbf{S}_K) | \mathbf{S}_{K-1} = \mathbf{s}\right],&& \\
         C_k(\mathbf{s}) &= \mathbb{E}\left[e^{-r\Delta t} \max \left( g(\mathbf{S}_{k+1}), C_{k+1}(\mathbf{S}_{k+1}) \right) | \mathbf{S}_k = \mathbf{s} \right] &&\text{ for } k = K-2:-1:0.
    \end{aligned}
\end{equation}
Given an approximation to $C_0(\mathbf{s})$, the price of Bermudan option can be obtained by
\begin{equation}\label{eq:goal}
    \mathcal{V}_0(\mathbf{S}_0) = \max( g(\mathbf{S}_0), C_0(\mathbf{S}_0)).
\end{equation}
The reformulation in terms of $C_k$ is crucial to the development of the numerical scheme.

Next we introduce the rotated log-price, which has independent components with Gaussian densities. We
denote the correlation matrix as $P = (\rho_{ij})_{d\times d}$, the volatility matrix $\Sigma$ as a diagonal matrix with
volatility $\sigma_i$ on the diagonal, and write the dividend yields as a vector $\bm{\delta} = [\delta_1, \dots, \delta_d]^\top$.
Then the log-price $\mathbf{X}(t)$ with each component defined by $X^i(t):= \ln(S^i(t)/S^i_0)$ follows a multivariate Gaussian distribution
\[ \mathbf{X}(t)\sim
\mathcal{N}\left(\left(r - \bm{\delta} -
\frac{1}{2}\Sigma^2\mathbf{1}\right)t, \Sigma P \Sigma^\top t\right),
\]
with $\mathbf{1} := [1,1,\dots,1]^\top\in\mathbb{R}^d$. The covariance matrix $\Sigma P \Sigma^\top$ admits the spectral decomposition $\Sigma P \Sigma^\top=Q^{\top}\Lambda Q$. Then the rotated log-price $\Tilde{\mathbf{X}}(t) := Q^\top \mathbf{X}(t)$ follows an independent Gaussian distribution
\[\Tilde{\mathbf{X}}(t)\sim
\mathcal{N}\left(Q^\top \left(r - \bm{\delta} - \frac{1}{2}\Sigma^2\mathbf{1}\right)t, \Lambda t\right).
\]
Therefore, to eliminate the correlation, we introduce the transformation
\begin{align*}
\Tilde{\mathbf{X}}(t) := Q^\top \ln(\mathbf{S}(t)./\mathbf{S}_0),
\end{align*}
and denote the inverse transformation by
\begin{align*}
    \phi(\Tilde{\mathbf{X}}(t)): = \mathbf{S}(t) = \mathbf{S}_0.*\exp(Q\Tilde{\mathbf{X}}(t)),
\end{align*}
where $./$ and $.*$ represent component-wise division and multiplication.

Finally, by defining $f_k(\mathbf{x}) := C_k(\phi(\mathbf{x}))$ as the continuation value function with respect to the rotated log-price for $\mathbf{x}\in \R^d$, we obtain the following dynamic programming procedure
\begin{equation} \label{eq: DP rotated log-price}
\begin{aligned}
f_{K-1}(\mathbf{x}) &= \mathbb{E}\left[e^{-r\Delta t}g( \phi(\Tilde{\mathbf{X}}_K)) \Big| \Tilde{\mathbf{X}}_{K-1} = \mathbf{x}\right],&&\\
 f_k(\mathbf{x}) &= \mathbb{E}\left[ e^{-r\Delta t}\max \left( g( \phi(\Tilde{\mathbf{X}}_{k+1})), f_{k+1}(\Tilde{\mathbf{X}}_{k+1})  \right) \Big| \Tilde{\mathbf{X}}_k = \mathbf{x} \right]&&  \text{ for } k = K-2:-1:0.
\end{aligned}
\end{equation}
The main objective of this work is to develop an efficient numerical method for solving problem \eqref{eq: DP rotated log-price}
with a large number of dimension $d$. The detailed construction is given in Section \ref{sec:method}.

\section{Methodologies}\label{sec:method}
In this section, we systematically develop a novel algorithm, based on quadrature and sparse grids polynomial interpolation (SGPI) \cite{barthelmann2000high} to solve problem \eqref{eq: DP rotated log-price} so that highly accurate results can be obtained for moderately large dimensions. This is achieved as follows.
First, we propose a mapping  $\psi$ \eqref{eq: mapping for unbounded} that transforms the domain from $\R^d$ to $\Omega:=(-1,1)^d$, and obtain problem
\eqref{eq: DP bounded variable} with the unknown function $F_{k}$ defined over the hypercube $\Omega$ in Section \ref{subsec:unbounded-bounded}. The mapping $\psi$ enables utilizing identical sparse grids for all time stepping $k = K-1:-1:0$, which greatly facilitates the computation, and
moreover, it avoids domain truncation and artificial boundary data when applying SGPI, which eliminates extra approximation errors.
However, the partial derivatives of $F_{k}$ may have boundary singularities, leading to low-efficiency of SGPI. To resolve this issue, we multiply $F_k$ with a bubble function \eqref{eq: bubble function}, and derive problem \eqref{eq: DP bubble} with unknown functions $u_k$ defined over the hypercube $\Omega$. Second, we present the SGPI of the unknown $u_k$ in Section \ref{subsec:sg}. Third and last, we provide several candidates for quadrature rules and summarize the algorithm in Section \ref{subsec:algorithm}, and analyze its complexity.

\subsection{Mapping for unbounded domains and bubble functions}\label{subsec:unbounded-bounded}
A direct application of the SGPI to problem \eqref{eq: DP rotated log-price} is generally involved since the problem is formulated on $\mathbb{R}^d$. For $d=1$, quadrature and interpolation-based schemes can be applied to problem \eqref{eq: DP rotated log-price} with domain truncation and suitable boundary conditions, e.g., payoff function. However, for $d\ge 2$, the exact boundary conditions of the truncated domain requires solving $(d-1)$-dimensional American option pricing problems \cite{nielsen2008penalty}. Therefore, the unboundedness of the domain $\mathbb{R}^d$ and the absence of natural boundary conditions pose great challenges to develop direct yet efficient interpolation for the continuation value function $f_k$, and there is an imperative need to develop a new approach to overcome the challenges.

Inspired by spectral methods on unbounded domains \cite{boyd2001chebyshev},
we propose the use of the logarithmic transformation $\psi: \R^d \to \Omega := (-1,1)^d$, defined by
\begin{equation} \label{eq: mapping for unbounded}
    \left\{\begin{aligned}
        & \mathbf{Z} = \psi(\Tilde{\mathbf{X}}) \text{ with each component } Z^i := \tanh(L \Tilde{X}^i) \in (-1,1), \\
        & \Tilde{\mathbf{X}} = \psi^{-1}(\mathbf{Z}) \text{ with each component } \Tilde{X}^i := L^{-1} \arctanh(Z^i)\in \mathbb{R},
    \end{aligned}\right.
\end{equation}
where $L>0$ is a scale parameter controlling the slope of the mapping. The logarithmic mapping $\psi$ is employed since the transformed points decay exponentially as they tend to infinity. Asymptotically the exponential decay rate matches that of the Gaussian distribution of the rotated log-price $\Tilde{\mathbf{X}}(t)$. Then by It\^{o}'s lemma, the new stochastic process $\mathbf{Z}(t) = [Z^1(t), \dots, Z^d(t)]^\top$ satisfies the stochastic differential equations
\begin{equation} \label{eq: SDE of Z}
    \mathrm{d}Z^i(t) = L\left( 1-(Z^i(t))^2 \right) \left( (\mu_i - \lambda_i L Z^i(t))\,\mathrm{d}t + \sqrt{\lambda_i} \,\mathrm{d}W^i(t) \right) \quad \text{ for } i=1,\dots,d,
\end{equation}
where $\bm{\mu} = [\mu_1,\dots,\mu_d]^\top = Q^\top \left(r - \bm{\delta} - \frac{1}{2}\Sigma^2\mathbf{1}\right)$, $\lambda_i$ are diagonal elements of $\Lambda$, and $W^i(t)$ are independent standard Brownian motions. Note that the drift and diffusion terms in \eqref{eq: SDE of Z} vanish on the boundary $\partial \Omega$. Thus, \eqref{eq: SDE of Z} fulfills the reversion condition \cite{zhu2003multi}, which implies $\mathbf{Z}(t)\in \Omega$ for $t>s$ provided $\mathbf{Z}(s)\in \Omega$.

Then we apply the mapping $\psi$ to the dynamic programming procedure \eqref{eq: DP rotated log-price}.
Let $F_k(\mathbf{z}) := f_k(\psi^{-1}(\mathbf{z})) = C_k\left(\phi (\psi^{-1}(\mathbf{z}))\right)$ be the continuation value function
of the bounded variable $\mathbf{z}$. Then problem \eqref{eq: DP rotated log-price} can be rewritten as,
for any $\mathbf{z}\in \Omega$,
\begin{equation} \label{eq: DP bounded variable}
    \begin{aligned}
         F_{K-1}(\mathbf{z}) &= \mathbb{E}\left[e^{-r\Delta t} g\left( \phi ( \psi^{-1}(\mathbf{Z}_K))\right) \big| \mathbf{Z}_{K-1} = \mathbf{z}\right],&& \\
        F_k(\mathbf{z}) &= \mathbb{E}\left[ e^{-r\Delta t}\max \left(g\left(\phi (\psi^{-1}(\mathbf{Z}_{k+1}))\right), F_{k+1}(\mathbf{Z}_{k+1})  \right) \big| \mathbf{Z}_k = \mathbf{z} \right]   &&\text{ for } k = K-2:-1:0.
    \end{aligned}
\end{equation}
Below we denote by $H(\cdot) := g\left(\phi(\psi^{-1}(\cdot))\right)$ the payoff function with respect to the bounded variable $\mathbf{z}$.

Note that problem \eqref{eq: DP bounded variable} is posed on the bounded domain $\Omega := (-1,1)^d$, $d\ge 1$,
which however remains challenging to approximate. First, the function $F_k: \Omega
\to \R$ may have singularities on the boundary $\partial\Omega$ due to the use of the mapping $\psi$. Second,
the Dirichlet boundary condition of $F_k$ is not identically zero, which is undesirable for controlling the
computational complexity of the algorithm, especially in high dimensions. Thus, we employ a bubble function
of the form
\begin{equation} \label{eq: bubble function}
    b(\mathbf{z}) = \prod_{i=1}^d (1-z_i^2)^{\beta}, \quad \mathbf{z}\in \overline{\Omega}:=[-1,1]^d,
\end{equation}
where the parameter $\beta>0$ controls the shape of $b(\mathbf{z})$. Note that $b(\mathbf{z})>0$ for $\mathbf{z}\in \Omega$ and $b(\mathbf{z})=0$ on $\partial \Omega$. Let
\begin{align*}
     u_k(\mathbf{z}) := F_k(\mathbf{z})b(\mathbf{z}),\quad k = 0,\ldots,K-1.
\end{align*}
Then the dynamic programming problem \eqref{eq: DP bounded variable} is equivalent to
\begin{equation} \label{eq: DP bubble}
    \begin{aligned}
         u_{K-1}(\mathbf{z}) &=  \mathbb{E}\left[e^{-r\Delta t}H(\mathbf{Z}_K) | \mathbf{Z}_{K-1} = \mathbf{z}\right] b(\mathbf{z}),&& \\
         u_k(\mathbf{z}) &= \mathbb{E}\left[ e^{-r\Delta t}\max \left( H(\mathbf{Z}_{k+1}), \frac{u_{k+1}(\mathbf{Z}_{k+1})}{b(\mathbf{Z}_{k+1})}  \right) \bigg| \mathbf{Z}_k = \mathbf{z} \right]b(\mathbf{z}) &&\text{ for } k = K-2:-1:0.
    \end{aligned}
\end{equation}
\begin{rem} \label{rem: 1}
The term $\frac{u_{k+1}(\mathbf{Z}_{k+1})}{b(\mathbf{Z}_{k+1})}$ appearing in \eqref{eq: DP bubble} is well-defined for
$\mathbf{Z}_{k+1} \in \Omega$. Nevertheless, since the bubble function $b(\mathbf{z})$ approaches zero as
$\mathbf{z}\to \partial\Omega$, in the numerical experiments, one should guarantee that $b(\cdot)$ evaluated on the computational grid will be greater than the machine epsilon, e.g., $\varepsilon = 2.2204\times 10^{-16}$ in MATLAB.
This will be established in Proposition \ref{prop:bubble}.
\end{rem}

\subsection{Approximation by sparse grids polynomial interpolation (SGPI)}\label{subsec:sg}
Next, we apply SGPI \cite{barthelmann2000high} to approximate the zero extension $\overline{u}_k:
\overline{\Omega} \to \R$ of $u_k$ iteratively backward in time. By choosing suitable bubble
functions $b(\mathbf{z})$, we shall prove in Theorem \ref{thm: bounded-mixed} below
the smoothness property of $\overline{u}_k$ up to order $r$. Thus the SGPI enjoys a geometrical
convergence rate with respect to the number of interpolation points $\Tilde{N}$, which depends only
on the dimension $d$ in a logarithm term, i.e. $\mathcal{O}(\Tilde{N}^{-r} (\log \Tilde{N})^{(r+1)
(d-1)})$ as stated in Proposition \ref{prop: error of sparse-interp}. In addition, thanks to the zero boundary condition
of $\overline{u}_k$, the required number of interpolation data is greatly reduced when compared with the full grids.

SGPI \cite{barthelmann2000high} (also called Smolyak approximation) is a powerful tool for constructing function approximations over a high-dimensional hypercube. Consider a function $f: \overline{\Omega} \to \R$. For $d = 1$, we denote by $X^\ell = \{x_1^\ell, \dots, x_{N_\ell}^\ell\}$ the set of nested Chebyshev-Gauss-Lobatto (CGL) points, with the nodes $x_j^\ell$ given by
\begin{equation*}
    x_j^\ell = \begin{cases}
        0 &\quad\text{ for }j=1 \text{ if }\ell = 1, \\
        \cos(\frac{(j-1)\pi}{N_\ell - 1}) &\quad \text{ for }j=1,\dots, N_\ell \text{ if }\ell \ge 2.
    \end{cases}
\end{equation*}
The cardinality of the set $X^\ell$ is
\begin{equation*}
    N_\ell = \begin{cases}
        1 &\quad \text{ if }\ell = 1, \\
        2^{\ell-1}+1 &\quad \text{ if }\ell\ge 2.
    \end{cases}
\end{equation*}
The polynomial interpolation $U^\ell f$ of $f$ over the set $X^\ell$ is defined as follows.
For $\ell = 1$, consider the midpoint rule, i.e., $(U^1 f)(x) = f(0)$. For $\ell\geq 2$, $U^\ell f$ is given by
\begin{equation*}
    (U^\ell f)(x) = \sum_{j=1}^{N_\ell} f(x_j^\ell) L_j^\ell(x), \quad \mbox{with } L_j^\ell(x) = \prod_{k=1, k\ne j}^{N_\ell} \frac{x-x_k}{x_j - x_k},
\end{equation*}
where $L_j^\ell$ are Lagrange basis polynomials.
Then we define the difference operator
\begin{equation*}
    \Delta^\ell f = (U^\ell - U^{\ell-1})f \quad \text{ with } U^0 f = 0.
\end{equation*}

For $d>1$, Smolyak's formula approximates a function $f:\overline{\Omega}\to\mathbb{R}$ by the interpolation operator
\begin{equation*}
    A(q,d) = \sum_{\Bell \in I(q,d)} \Delta^{\ell_1} \otimes \dots \otimes \Delta^{\ell_d},
\end{equation*}
with the index set $I(q,d):= \{\Bell \in \mathbb{N}^d: |\Bell| \le q\}$ and  $|\Bell| = \ell_1 + \dots + \ell_d$ \cite{barthelmann2000high}.
Equivalently, the linear operator $A(q,d)$ can be represented as \cite[Lemma 1]{wasilkowski1995explicit}
\begin{equation} \label{eq: SG interp}
    A(q,d) = \sum_{\Bell \in P(q,d)} (-1)^{q-|\Bell|}\binom{d-1}{q-|\Bell|} U^{\ell_1} \otimes \dots \otimes U^{\ell_d},
\end{equation}
with the index set $P(q,d):= \{\Bell \in \mathbb{N}^d: q-d+1 \le |\Bell| \le q\}$, where
the tensor product of the univariate interpolation operators is defined by
\begin{equation*}
    (U^{\ell_1} \otimes \dots \otimes U^{\ell_d})(f) = \sum_{j_1=1}^{N_{\ell_1}} \dots \sum_{j_d=1}^{N_{\ell_d}}f(x_{j_1}^{\ell_1}, \dots, x_{j_d}^{\ell_d}) (L_{j_1}^{\ell_1} \otimes \dots \otimes  L_{j_d}^{\ell_d}),
\end{equation*}
i.e., multivariate Lagrange interpolation. With the set $X^{\ell_i}$ (i.e., one-dimensional nested CGL points), the
formula \eqref{eq: SG interp} indicates that computing $A(q,d)(f)$ only requires function evaluations on the sparse grids
\begin{equation*}
    H(q,d) = \bigcup_{\Bell \in P(q,d)} X^{\ell_1} \times \dots \times X^{\ell_d}.
\end{equation*}
We denote the cardinality of $H(q,d)$ by $\Tilde{N}_{CGL}(q,d)$. Usually, the interpolation level $L_I\in \mathbb{N}_0$ is defined by
$$ L_I := q-d. $$
Then for fixed $L_I$ and $d\to \infty$, the following asymptotic estimate of $\Tilde{N}_{CGL}(L_I+d,d)$ holds \cite{novak1999simple}
\begin{equation} \label{eq: cardinality of N_CGL}
    \Tilde{N}_{CGL}(L_I+d,d) \approx \frac{2^{L_I}}{L_I!}d^{L_I}. 
\end{equation}

The sparse grid has much fewer grid points than the full grid generated by the tensor product. Furthermore, in high-dimensional hypercube, a significant number of sparse grids lie on the boundary. We will compare the number of CGL sparse grids
$\Tilde{N}_{CGL}(L_I+d,d)$ and the number of inner sparse grids $N$ in Section \ref{subsec:algorithm}.
In Remark \ref{rem: 1}, we require that $b(\cdot)$ evaluated on the computational grid be greater than the machine epsilon. Note that for all inner sparse grids $\{\mathbf{z}^n\}_{n=1}^N$ of the interpolation level $L_I$, each coordinate $z^n_j$ satisfies
\begin{equation} \label{eq: estimate interpolation points}
    -\cos\left(\frac{\pi}{2^{\ell_j-1}}\right) \le z^n_j \le \cos\left(\frac{\pi}{2^{\ell_j-1}}\right),
\end{equation}
where $\ell_1 + \dots + \ell_d \le L_I+d$.

\subsection{Numerical algorithm}\label{subsec:algorithm}
We interpolate the function $\overline{u}_k: \overline{\Omega} \to \R$ on CGL sparse grids in the dynamic programming \eqref{eq: DP bubble}, where the interpolation data can be formulated as high-dimensional integrals. Since $\overline{u}_k(\mathbf{z}) = 0$ for $\mathbf{z}\in \partial \Omega$, only function evaluations on the inner sparse grids are required. This greatly reduces the computational complexity, especially for large $d$. Indeed, since $\Tilde{\mathbf{X}}_{k+1} = \Tilde{\mathbf{X}}_k + \mathbf{Y}$ with $\mathbf{Y} \sim \mathcal{N}(Q^\top (r - \bm{\delta} - \frac{1}{2}\Sigma^2\mathbf{1})\Delta t, \Lambda \Delta t)$,
\begin{equation} \label{eq: pts next step}
    \mathbf{Z}_{k+1} = \psi (\psi^{-1}(\mathbf{Z}_k) + \mathbf{Y}).
\end{equation}
For a fixed interpolation knot $\mathbf{z} \in \Omega$, we denote
\begin{equation} \label{eq: def of integrand}
    v_{k+1}^{\mathbf{z}}(\mathbf{Y}) = \max \left( H\big(\psi (\psi^{-1}(\mathbf{z}) + \mathbf{Y}) \big), \frac{u_{k+1}}{b} \big( \psi (\psi^{-1}(\mathbf{z}) + \mathbf{Y}) \big) \right),
\end{equation}
and the probability density of $\mathbf{Y}$ as
$\rho(\mathbf{y}) = \prod_{i=1}^d \rho_i(y_i)$.
Following \eqref{eq: DP bubble}, the interpolation   data $u_k(\mathbf{z})$ is given by
\begin{equation*}
    \begin{split}
        u_k(\mathbf{z}) = e^{-r\Delta t}\mathbb{E}[v_{k+1}^{\mathbf{z}}(\mathbf{Y})] b(\mathbf{z}) = e^{-r\Delta t}b(\mathbf{z}) \int_{\R^d} v_{k+1}^{\mathbf{z}}(\mathbf{y}) \rho(\mathbf{y}) \,\mathrm{d}\mathbf{y},
    \end{split}
\end{equation*}
where the last integral can be computed by any high-dimensional quadrature methods, e.g., Monte Carlo (MC), quasi-Monte Carlo (QMC) method, or sparse grid quadrature.
\begin{enumerate}
    \item The \textbf{Monte Carlo method} approximates the integral by averaging random samples of the integrand
    \begin{equation} \label{eq: MC}
        \int_{\R^d} v_{k+1}^{\mathbf{z}}(\mathbf{y}) \rho(\mathbf{y}) \,\mathrm{d}\mathbf{y} \approx \frac{1}{M}\sum_{m=1}^M v_{k+1}^{\mathbf{z}}(\mathbf{y}^m),
    \end{equation}
    where $\{\mathbf{y}^m\}_{m=1}^M$ are independent and identically distributed (i.i.d.) random samples drawn from the distribution $\rho(\mathbf{y})$.

    \item The \textbf{Quasi-Monte Carlo method} takes the same form as \eqref{eq: MC}, but $\{\mathbf{y}^m\}_{m=1}^M$ are the transformation of QMC points. By changing variables $\mathbf{x} = \Phi(\mathbf{y})$ with $\Phi$ being the cumulative density function (CDF) of $\mathbf{Y}$, we have
    \begin{align}
        \int_{\R^d} v_{k+1}^{\mathbf{z}}(\mathbf{y}) \rho(\mathbf{y}) \,\mathrm{d}\mathbf{y} &= \int_{[0,1]^d} v_{k+1}^{\mathbf{z}}\left(\Phi^{-1}(\mathbf{x}) \right)\,\mathrm{d}\mathbf{x}\nonumber \\
        &\approx \frac{1}{M}\sum_{m=1}^M v_{k+1}^{\mathbf{z}}\left(\Phi^{-1}(\mathbf{x}^m) \right)= \frac{1}{M}\sum_{m=1}^M v_{k+1}^{\mathbf{z}}(\mathbf{y}^m), \label{eq: QMC}
    \end{align}
    where $\{\mathbf{x}^m\}_{m=1}^M$ are QMC points taken from a low-discrepancy sequence, e.g., Sobol sequence and sequences generated by the lattice rule. Note that other transformations $\Phi$ are also available for designing QMC approximation \cite{kuo2016practical}.

    \item The \textbf{Sparse grid quadrature} approximates the integral based on a combination of tensor products of univariate quadrature rule. For the integration with Gaussian measure, several types of sparse grids are available, including Gauss-Hermite, Genz-Keister, and weighted Leja points \cite{piazzola.tamellini:SGK}. Let $(\omega_m,\mathbf{y}^m)_{m=1}^M$ be quadrature weights and points associated with an anisotropic Gaussian distribution of $\mathbf{Y}$. Then the integral is computed by
    \begin{equation} \label{eq: SG quadrature}
        \int_{\R^d} v_{k+1}^{\mathbf{z}}(\mathbf{y}) \rho(\mathbf{y}) \,\mathrm{d}\mathbf{y} \approx \sum_{m=1}^M \omega_m v_{k+1}^{\mathbf{z}}(\mathbf{y}^m).
    \end{equation}
\end{enumerate}

By employing the transformation between the asset price $\mathbf{S}$ and the bounded variable $\mathbf{z}$, the sparse grid interpolation points and the quadrature points are shown in Fig. \ref{fig: interp_quad_pts}.

\begin{figure}[H]
    \centering
    \begin{subfigure}[b]{0.4\textwidth}
        \centering
        \includegraphics[width = .95\textwidth]{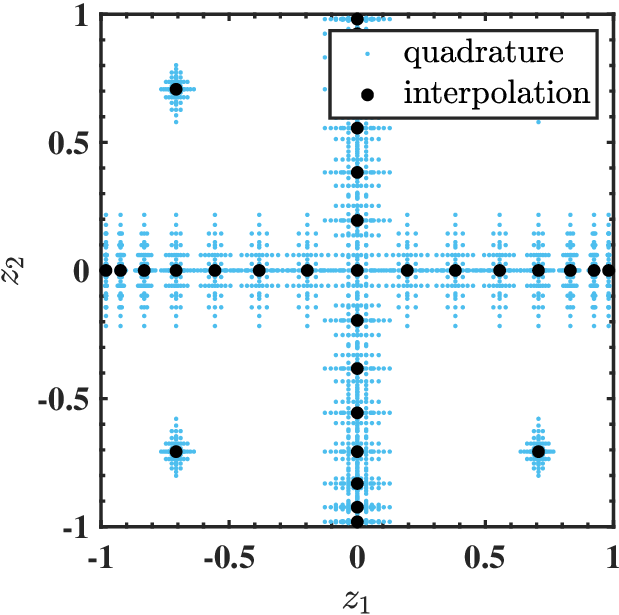}
        \caption{}
        \label{fig: interp_quad_pts_a}
    \end{subfigure}
    ~
    \begin{subfigure}[b]{0.4\textwidth}
        \centering
        \includegraphics[width = \textwidth]{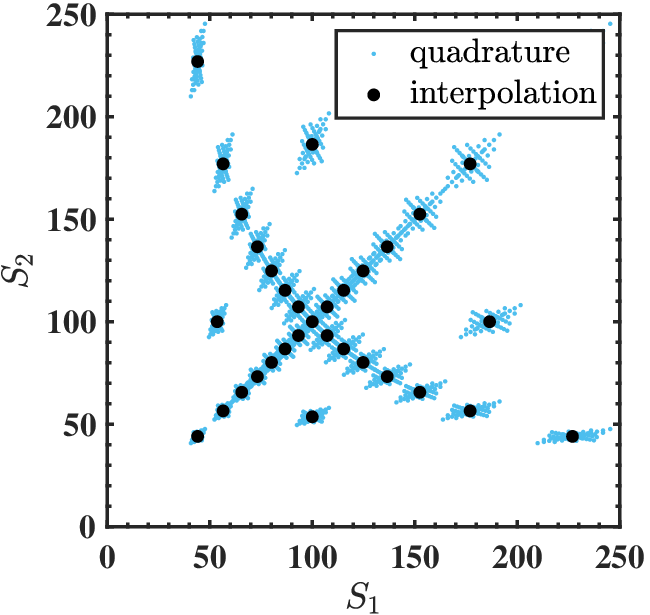}
        \caption{}
        \label{fig: interp_quad_pts_b}
    \end{subfigure}
    \caption{A schematic illustration of interpolation and quadrature points in $d = 2$; (a) in the domain of the mapped bounded variable $\mathbf{z}$; (b) in the domain of the price of underlying assets $\mathbf{S}$.}
    \label{fig: interp_quad_pts}
\end{figure}

Now we can describe the procedure of iteratively interpolating the function $\overline{u}_k$ in Algorithm \ref{alg: SGSG}.

\begin{algorithm}[hbt]
    \caption{Pricing American option with $d$ assets} \label{alg: SGSG}
    \KwData{Market parameters: $ \mathbf{S}_0, r, \delta_i, \sigma_i, P $\;
    \hspace{1cm} Option parameters: $\kappa, T$, option type\;
    \hspace{1cm} Transform and discretization parameters: $L,K,L_I$.
    }
    \KwResult{Option price $V_0$}
    Generate CGL sparse grids in $\overline{\Omega}$ \;
    Find grids in $\Omega$, which are denoted by $\{\mathbf{z}^n\}_{n=1}^N$ \;
    Generate quadrature points $\{\mathbf{y}^m\}_{m=1}^M$ and weights $\{w_m\}_{m=1}^M$ with Gaussian density $ \mathcal{N}(Q^\top (r - \bm{\delta} - \frac{1}{2}\Sigma^2\mathbf{1})\Delta t, \Lambda \Delta t)$ \;
    \nonl \textbf{Dynamic programming:} \\
    Compute the payoff $H_{n,m} = H(\psi(\psi^{-1}(\mathbf{z}^n) + \mathbf{y}^m))$ \;
    Set the terminal value as $V_{n,m} = H_{n,m}$ \;
    \For{$k = K-1:-1:0$}{
       $u(\mathbf{z}^n) = \alpha \sum_{m=1}^{M} w_m V_{n,m}b(\mathbf{z}^n)$ where $\alpha = e^{-r\Delta t}$ is the discount factor \;
       \If{$k==0$}{
       $u_0 = A(L_{I}+d,d)(u)(\mathbf{0})$, break \;
       }
       $u_{n,m} = A(L_I+d,d)(u)\left(\psi(\psi^{-1}(\mathbf{z}^n) + \mathbf{y}^m)\right)$  \Comment*[r]{Compute in parallel}
       Update $V_{n,m} = \max\big( H_{n,m}, u_{n,m}/b(\psi(\psi^{-1}(\mathbf{z}^n) + \mathbf{y}^m)) \big)$\;
    }
    $V_0 = \max\left( H(\mathbf{0}), \frac{u_0}{b(\mathbf{0})}\right)$.
\end{algorithm}

Followed by the Remark \ref{rem: 1}, the next result gives a sufficient condition on the well-definedness of the algorithm.
\begin{prop}\label{prop:bubble}
Let $\varepsilon$ be the machine epsilon. Assume that each coordinate $y_j^m$ of the sampling or quadrature points $\{\mathbf{y}^m\}_{m=1}^M$ of the random variable $\mathbf{Y} \sim \mathcal{N}(Q^\top (r - \bm{\delta} - \frac{1}{2}\Sigma^2\mathbf{1})\Delta t, \Lambda \Delta t)$ satisfies
\begin{equation} \label{eq: assumption}
    \max_{m=1,2,\dots, M} |y^m_j| \le C \sqrt{\lambda_j \Delta t},
\end{equation}
where $\lambda_j$ is the $j$-th diagonal element of $\Lambda$, and $C$ is a constant. If $\beta = 1$ and  $L_I, L$ and $\Delta t$ are chosen such that
    \begin{equation*}
        \frac{4^{L_I+d}}{\pi^{2d}} \exp \Big( 2CL\sqrt{\Delta t}\sum_{j=1}^d \sqrt{\lambda_j} \Big) \le \frac{1}{\varepsilon},
    \end{equation*}
    then for all sampling or quadrature points $\mathbf{Z}_{k+1}^{n,m}$ of $\mathbf{Z}_{k+1}$, $n=1,2,\dots,N, m = 1,2,\dots,M$, we have
    \begin{equation*}
        b(\mathbf{Z}_{k+1}^{n,m}) > \varepsilon.
    \end{equation*}
\end{prop}
\begin{proof}
Using \eqref{eq: pts next step}, we have
$\mathbf{Z}_{k+1}^{n,m} = \psi(\psi^{-1}(\mathbf{z}^n) + \mathbf{y}^m)$,
where $\{\mathbf{z}^n\}_{n=1}^N$ are inner sparse grid interpolation points and $\{\mathbf{y}^m\}_{m=1}^M$ are sampling or quadrature points of $\mathbf{Y}$. Thus the $j$-th coordinate of $\mathbf{Z}_{k+1}^{n,m}$ is given by
    \begin{align*}
        Z_{k+1}^{n,m,j} = 1-\frac{2}{1+\eta_j^{n,m}}, \quad \mbox{with }
\eta_j^{n,m}:=\frac{1+z^n_j}{1-z^n_j}\exp(2Ly^m_j).
    \end{align*}
For $\beta = 1$, to ensure
$b(\mathbf{z}) = \prod_{j=1}^d (1-z_j^2) > \varepsilon$,
it suffices to prove $\prod_{j=1}^d (1-|z_j|) >\varepsilon$.
Without loss of generality, consider the case $Z_{k+1}^{n,m,j} \to 1$, i.e., $\eta_j^{n,m}\to+\infty$. Clearly, we have
\begin{equation*}
  \prod_{j=1}^d (1-Z_{k+1}^{n,m,j}) = \prod_{j=1}^d \frac{2}{1+\eta_j^{n,m}} > \varepsilon \quad \Leftrightarrow \quad
  \prod_{j=1}^d \left( 1+\eta_j^{n,m} \right) < \frac{2^d}{\varepsilon}.
  \end{equation*}
Noting that $\eta_j^{n,m}\to+\infty$, the binomial expansion implies that it suffices to have
$\prod_{j=1}^d \eta_j^{n,m} <\varepsilon^{-1}$. Using the relationship
    \begin{equation*}
        \frac{1+\cos x}{1-\cos x} = \frac{1}{\tan^2 \frac{x}{2}} \le \frac{4}{x^2} \quad \text{ as } x\to 0^+,
    \end{equation*}
the inequality \eqref{eq: estimate interpolation points}, and the assumption \eqref{eq: assumption}, we obtain
    \begin{equation*}
        \begin{split}
            \prod_{j=1}^d \left( \frac{1+z^n_j}{1-z^n_j}\exp(2Ly^m_j) \right) &\le \prod_{j=1}^d \frac{1+\cos\left(\frac{\pi}{2^{\ell_j-1}}\right)}{1-\cos\left(\frac{\pi}{2^{\ell_j-1}}\right)} \exp(2CL \sqrt{\lambda_j \Delta t}) \\
            &\le \left(\frac{4}{\pi^2}\right)^d 4^{\sum_{j=1}^d\ell_j - d} \exp \left( 2CL\sqrt{\Delta t}\sum_{j=1}^d \sqrt{\lambda_j} \right).
        \end{split}
    \end{equation*}
Since $\sum_{j=1}^d\ell_i\leq L_I+d$, this proves the desired assertion.
\end{proof}

Last, we discuss the computational complexity of the algorithm. By introducing the bubble function, the interpolation functions $\overline{u}_k$ have zero boundary values. Thus, at each time step, we require $N$ evaluations of $u_k$ only on the inner sparse grids, where each evaluation is approximated by equation \eqref{eq: MC}, \eqref{eq: QMC} or \eqref{eq: SG quadrature} with $M$ sampling or quadrature points. The number of inner sparse grids $N$ of level $L_I = 5$ in dimension $d$ is listed in Table \ref{tab: num SG}, which also shows the numbers for tensor-product full grids $\Tilde{N}_{full}$ and CGL sparse grids with boundary points $\Tilde{N}_{CGL}$. Unlike full grids, the number of sparse grids does not increase exponentially as the dimension increases. In particular, the inner sparse grids account for less than three percent of the CGL sparse grids in $d=10$. This represents a dramatic reduction of the evaluation points.

\begin{figure}
\begin{minipage}[hbt!]{.5\textwidth}
    \hspace{2cm}
    \begin{tabular}{c c c c}
        \toprule
        $d$ & $\Tilde{N}_{full}$ & $\Tilde{N}_{CGL}$ & $N$ \\
        \midrule
        2 & 961    & 145   & 81 \\
        3 & 29791  & 441   & 151 \\
        4 & 923521 & 1105  & 241 \\
        5 & 2.8629e+7   & 2433  & 351 \\
        6 & 8.8750e+8   & 4865  & 481 \\
        7 & 2.7513e+10  & 9017  & 631 \\
        8 & 8.5289e+11  & 15713 & 801 \\
        9 & 2.6440e+13  & 26017 & 991 \\
        10 & 8.1963e+14 & 41265 & 1201 \\
        \bottomrule
    \end{tabular}
\end{minipage}
\begin{minipage}[hbt!]{.5\textwidth}
    \includegraphics[width=0.7\textwidth]{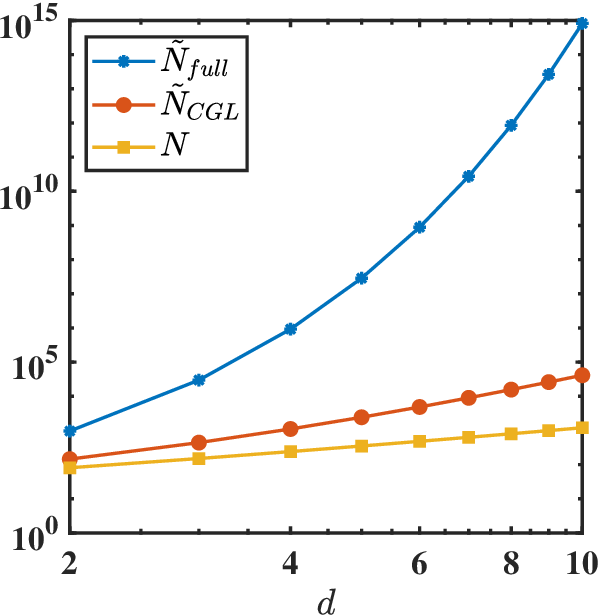}
\end{minipage}
\caption{The number of tensor-product full grids $\Tilde{N}_{full}$, the number of Chebyshev-Gauss-Lobatto sparse grids $\Tilde{N}_{CGL}$, and the number of inner sparse grids $N$ in the $d$-dimensional hypercube with level $L_I= 5$.}
\label{tab: num SG}
\end{figure}

\section{Smoothness analysis}\label{sec:analysis}
In this section, we analyze the smoothness of the function $u_k$, in order to justify the use of SGPI, thereby providing solid theoretical underpinnings of its excellent performance. 

First we list several useful notations. Let $\bm{\alpha} = (\alpha_1, \dots, \alpha_d) \in \mathbb{N}_0^d$ be the standard multi-index with $|\bm{\alpha}| = \alpha_1 + \dots + \alpha_d$. Then  $\bm{\alpha} + \bm{\gamma} := (\alpha_1 + \gamma_1, \dots, \alpha_d + \gamma_d)$, $\bm{\alpha}! := \prod_{j=1}^d \alpha_j!$,
and $\bm{\gamma} \preceq \bm{\alpha}$ denotes that each component of the multi-index $\bm{\gamma} = (\gamma_1,\dots,\gamma_d)$ satisfies $\gamma_j \le \alpha_j$. We define the differential operator $D^{\bm{\alpha}}$ by
$D^{\bm{\alpha}}f := \frac{\partial^{|\bm{\alpha}|}f}{\prod_{j=1}^d \partial x_j^{\alpha_j}}$.
For an open set $D \subset \mathbb{R}^d$ and $r\in \mathbb{N}$, the space $\mathcal{C}^r(\overline{D})$ denotes the space of functions with their derivatives of orders up to $r$ being continuous on the closure $\overline{D}$ of $D$, i.e.,
\begin{equation*}
    \mathcal{C}^r(\overline{D}) = \{f: \overline{D} \to \R \mid D^{\bm{\alpha}}f \text{ continuous if }|\bm{\alpha}|\le r \}.
\end{equation*}
Specially, $\mathcal{C}^r(\overline{\mathbb{R}^d})$ consists of functions $f\in \mathcal{C}^r(\mathbb{R}^d)$ such that $D^{\bm{\alpha}}f$ is bounded and uniformly continuous on $\mathbb{R}^d$ for all $|\bm{\alpha}|\le r$ \cite{adams2003sobolev}. We also define $\mathcal{C}^\infty(\overline{\mathbb{R}^d})$ as the intersection of all $\mathcal{C}^r(\overline{\mathbb{R}^d})$ for $r\in \mathbb{N}$, i.e., $\mathcal{C}^\infty(\overline{\mathbb{R}^d}):=\bigcap_{r=1}^{\infty}\mathcal{C}^r(\overline{\mathbb{R}^d})$.

The analysis of SGPI employs the space of functions on $\overline{\Omega}:=[-1,1]^d$ with bounded mixed derivatives \cite{bungartz2004sparse}. Let $F_d^r(\overline{\Omega})$ be the set
of all functions $f:\overline{\Omega}\to \R$ such that $D^{\bm{\alpha}}f$ is continuous for all $\bm{\alpha}\in \mathbb{N}_0^d$ with $\alpha_i\le r$ for all $i$, i.e.,
\begin{equation} \label{eq: function space}
    F_d^r(\overline{\Omega}) := \{f:\overline{\Omega} \to \R \mid D^{\bm{\alpha}} f \text{ continuous if }\alpha_i\le r \text{ for all }i\}.
\end{equation}
We equip the space $F_d^r(\overline{\Omega})$ with the norm
$\|f\|_{F_d^r(\overline{\Omega})} := \max\{\|D^{\bm{\alpha}} f\|_{L^\infty(\overline{\Omega})} \mid \bm{\alpha} \in \mathbb{N}_0^d, \alpha_i\le r\}$.

For the sake of completeness, we present in Proposition \ref{prop: error of sparse-interp} the interpolation error using SGPI described in Section \ref{subsec:sg}.
\begin{prop}[{\cite[Theorem 8, Remark 9]{barthelmann2000high}}]
\label{prop: error of sparse-interp}
For $f\in F_d^r(\overline{\Omega})$, there exists a constant $c_{d,r}$ depending only on $d$ and $r$ such that
\begin{equation*}
    \|A(q,d)(f) - f\|_{L^\infty(\overline{\Omega})} \le c_{d,r} \cdot \Tilde{N}^{-r} \cdot (\log \Tilde{N})^{(r+1)(d-1)} \|f\|_{F_d^r(\overline{\Omega})},
\end{equation*}
where $\Tilde{N} = \Tilde{N}_{CGL}(q,d)$ is the number of CGL sparse grids.
\end{prop}

To provide theoretical guarantees of applying SGPI, we next prove $\overline{u}_k \in F_d^r(\overline{\Omega})$ in Theorem \ref{thm: bounded-mixed}.
This result follows by Lemma \ref{lem: 1} \footnote{At a first glance of \eqref{eq: continuation value as integral}, the $\mathcal{C}^\infty(\R^d)$ smoothness of $f_{K-1}$ follows directly by the fact that convolution smooths out the payoff function. However, the payoff function $G$ of the rotated log-price is not in  $L^1(\R^d)$, neither does the density function have compact support. Therefore, we provide a proof of the $\mathcal{C}^\infty(\overline{\R^d})$ regularity in Lemma \ref{lem: 1} mainly using the dominated convergence theorem.} and Lemma \ref{lem: 2}.
\begin{lem}[$f_{K-1} \in \mathcal{C}^\infty(\overline{\R^d})$] \label{lem: 1}
Let $g$ be the payoff of a put option. Then the continuation value function $f_{K-1}$ defined in \eqref{eq: DP rotated log-price} is infinitely differentiable, bounded and uniformly continuous with all its derivatives up to the order $r$ for any $r\in \mathbb{N}$, i.e., $f_{K-1} \in \mathcal{C}^\infty(\overline{\R^d})$.
\end{lem}
\begin{proof}
Consider the conditional expectation without the discount factor $e^{-r\Delta t}$:
\begin{equation} \label{eq: continuation value as integral}
        f(\mathbf{x}) = \mathbb{E}[G(\Tilde{X}_K) | \Tilde{X}_{K-1} = \mathbf{x}]
       = \int_{\R^d} G(\mathbf{y}) p(\mathbf{x},\mathbf{y})\,\mathrm{d}\mathbf{y},
\end{equation}
where $G(\cdot) = g(\phi(\cdot))$ is the payoff with respect to the rotated log-price, and $p(\mathbf{x},\mathbf{y})$ is the density of the Gaussian distribution $\mathcal{N}(\mathbf{x} + Q^\top (r - \bm{\delta} - \frac{1}{2}\Sigma^2\mathbf{1})\Delta t, \Lambda \Delta t)$, that is,
\begin{equation} \label{eq: transition density}
    p(\mathbf{x},\mathbf{y}) = \frac{1}{(2\pi)^{d/2}\sqrt{ \det( \Lambda  )\Delta t}}\exp \left( -\frac{1}{2 \Delta t}(\mathbf{y} - \mathbf{x} - \bm{\mu})^\top \Lambda^{-1}  (\mathbf{y} - \mathbf{x} - \bm{\mu}) \right)
\end{equation}
with $\bm{\mu} = Q^\top (r - \bm{\delta} - \frac{1}{2}\Sigma^2\mathbf{1})\Delta t$.
Since $g$ is the payoff of a put option, $G$ is bounded in $\R^d$. Let
\begin{equation*}
    P(\mathbf{x}) := \frac{1}{(2\pi)^{d/2}\sqrt{ \det( \Lambda  )\Delta t}}\exp \left( -\frac{1}{2 \Delta t}(\mathbf{x} + \bm{\mu})^\top \Lambda^{-1}  (\mathbf{x} + \bm{\mu}) \right).
\end{equation*}
Then $P\in L^1(\R^d)$ with $\|P\|_1= \int_{\R^d}P(\mathbf{x})\,\mathrm{d}\mathbf{x} = 1$. The representation \eqref{eq: continuation value as integral} is equivalent to
\begin{equation*}
    f(\mathbf{x}) = G*P(\mathbf{x}),
\end{equation*}
where $*$ denotes the convolution operator. Then an application of \cite[Proposition 8.8]{folland1999real} implies that $f$ is bounded and uniformly continuous in $\R^d$, and
\begin{equation} \label{eq: boundedness of f}
    \|f\|_{L^\infty(\R^d)} \le \|G\|_{L^\infty(\R^d)} \|P\|_{L^1(\R^d)} = \|G\|_{L^\infty(\R^d)}.
\end{equation}
Next, we show that $f$ has the bounded first order partial derivatives for all $i\in \{1,2,\dots,d\}$, and
\begin{equation} \label{eq: 1st derivative}
    \frac{\partial f}{\partial x_i}(\mathbf{x}) = G*\frac{\partial P}{\partial x_i}(\mathbf{x}).
\end{equation}
For any fixed $\mathbf{x}_0\in \R^d$, $i\in \{1,2,\dots,d\}$,
\begin{equation*}
    \frac{\partial p}{\partial x_i}(\mathbf{x}_0, \mathbf{y}) = \lim_{\mathbf{x}_n \to \mathbf{x}_0} q_{i,n}(\mathbf{y}),\quad \text{with } q_{i,n}(\mathbf{y}) = \frac{p(\mathbf{x}_n, \mathbf{y}) - p(\mathbf{x}_0, \mathbf{y})}{x_{i,n} - x_{i,0}}.
\end{equation*}
Since the Gaussian density $p(\mathbf{x}, \mathbf{y})$ has bounded partial derivatives $\frac{\partial p}{\partial x_i}(\mathbf{x},\mathbf{y})$ for all $\mathbf{x}, \mathbf{y}$ and $i\in \{1,2,\dots, d\}$, then
\begin{equation} \label{eq: derivative in limit}
    \begin{split}
        \frac{\partial f}{\partial x_i}(\mathbf{x}_0) &= \lim_{\mathbf{x}_n \to \mathbf{x}_0} \frac{f(\mathbf{x}_n) - f(\mathbf{x}_0) }{x_{i,n} - x_{i,0}} \\
        &= \lim_{\mathbf{x}_n \to \mathbf{x}_0} \int_{\R^d} G(\mathbf{y}) \frac{ p(\mathbf{x}_n,\mathbf{y}) - p(\mathbf{x}_0, \mathbf{y}) }{x_{i,n} - x_{i,0}} \,\mathrm{d}\mathbf{y} \\
        &= \lim_{\mathbf{x}_n \to \mathbf{x}_0} \int_{\R^d} G(\mathbf{y}) q_{i,n}(\mathbf{y}) \,\mathrm{d}\mathbf{y}.
    \end{split}
\end{equation}
As $G$ is bounded and for fixed $\mathbf{x}$, the function $p(\mathbf{x},\mathbf{y})$ decays as $\mathcal{O}(e^{-\|\mathbf{y}\|_2^2})$ at infinity, the integrand $G(\mathbf{y})q_{i,n}(\mathbf{y})$ is bounded by some Lebesgue integrable function for all $n$. Hence, by the dominated convergence theorem, one can interchange the limit and integral in the last equation of \eqref{eq: derivative in limit} and thus
\begin{equation*}
    \frac{\partial f}{\partial x_i}(\mathbf{x}_0) = \int_{\R^d} G(\mathbf{y}) \frac{\partial p}{\partial x_i}(\mathbf{x}_0,\mathbf{y}) \,\mathrm{d}\mathbf{y} = G*\frac{\partial P}{\partial x_i}(\mathbf{x}_0).
\end{equation*}
By \cite[Proposition 8.8]{folland1999real}, we obtain that $\frac{\partial f}{\partial x_i}$ is bounded and uniformly continuous in $\R^d$, and
\begin{equation*}
    \left\|\frac{\partial f}{\partial x_i}\right\|_{L^\infty(\R^d)} \le \left\|G\right\|_{L^\infty(\R^d)} \left\|\frac{\partial P}{\partial x_i}\right\|_{L^1(\R^d)}.
\end{equation*}
Using the dominated convergence theorem, we derive the bounded second order partial derivatives
\begin{equation*}
    \frac{\partial^2 f}{\partial x_j \partial x_i}(\mathbf{x}) = G*\frac{\partial^2 P}{\partial x_j \partial x_i}(\mathbf{x}),
\end{equation*}
since for fixed $\mathbf{x}$, the function
\begin{equation*}
    \frac{\partial p}{\partial x_i}(\mathbf{x}, \mathbf{y}) = -\frac{(x_i - y_i + \mu_i)}{\lambda_i \Delta t} p(\mathbf{x},\mathbf{y})
\end{equation*}
decays as $\mathcal{O}(\mathbf{y} e^{-\|\mathbf{y}\|_2^2})$ at infinity. By \cite[Proposition 8.8]{folland1999real}, we obtain $\frac{\partial^2 f}{\partial x_j \partial x_i}$ is bounded and uniformly continuous on $\R^d$, and
\begin{equation*}
    \left\|\frac{\partial^2 f}{\partial x_j \partial x_i}\right\|_{L^\infty(\R^d)} \le  \left\|G\right\|_{L^\infty(\R^d)}\left \|\frac{\partial^2 P}{\partial x_j \partial x_i}\right\|_{L^1(\R^d)}.
\end{equation*}
Repeating the argument yields the boundedness and uniform continuity of the partial derivatives
of $f$ up to the order $r$ for any $r\in \mathbb{N}$. Therefore, $f_{K-1} \in \mathcal{C}^\infty(\overline{\R^d})$.
\end{proof}

\begin{lem}[$f_k\in \mathcal{C}^\infty(\overline{\R^d})$ for $k=0,1, \dots, K-1$] \label{lem: 2}
Let $g$ be the payoff of a put option. Then the continuation value functions $\{f_k\}_{k=0}^{K-1}$ defined in \eqref{eq: DP rotated log-price} are infinitely differentiable, bounded and uniformly continuous with all its derivatives up to the order $r$ for any $r\in \mathbb{N}$, i.e., $f_k\in \mathcal{C}^\infty(\overline{\R^d})$ for $k=0,1, \dots, K-1$.
\end{lem}
\begin{proof}
Let the value at $t_{k+1}$ be
\begin{equation*}
    V_{k+1}(\mathbf{y}) = \max\left( G(\mathbf{y}), f_{k+1}(\mathbf{y}) \right), \quad \text{ for } k = 0,1,\dots, K-2,
\end{equation*}
where $G(\cdot) = g(\phi(\cdot))$ is bounded for a put option. By \eqref{eq: boundedness of f}, $f_{K-1}$ is bounded. Hence, $V_{K-1}$ is bounded in $\R^d$. Using the argument of Lemma \ref{lem: 1}, we obtain
\begin{equation*}
    f_{K-2}(\mathbf{x}) = e^{-r\Delta t}\mathbb{E}[V_{K-1}(\Tilde{\mathbf{X}}_{K-1}) | \Tilde{\mathbf{X}}_{K-2} = \mathbf{x}] = e^{-r\Delta t} V_{K-1}*P(\mathbf{x}) 
\end{equation*}
is infinitely differentiable, bounded and uniformly continuous with all its derivatives up to the order $r$ for any $r\in \mathbb{N}$.
Since $\|f_{K-2}\|_{L^\infty(\R^d)} \le e^{-r\Delta t} \|V_{K-1}\|_{L^\infty(\R^d)} \|P\|_{L^1(\R^d)}$, the value $V_{K-2}$ is bounded in $\R^d$. Similarly, we can obtain $f_k\in \mathcal{C}^\infty(\overline{\R^d})$ for all $k=0,1, \dots, K-1$.
\end{proof}

Using the smoothness of the continuation value functions in $\R^d$, now we can establish the smoothness of the extended interpolation function $\overline{u}_k$ in the bounded domain $\overline{\Omega} = [-1,1]^d$.
\begin{thm}[$\overline{u}_k \in F_d^r(\overline{\Omega}) \text{ for } k=0,1,\dots,K-1$]\label{thm: bounded-mixed}
Let $g$ be the payoff of a put option. Let $u_k: \Omega \to \R$ be the function defined in \eqref{eq: DP bubble} with $u_k(\mathbf{z}) = f_k(\psi^{-1}(\mathbf{z}))b(\mathbf{z})$. Let $b(\cdot)$ be the bubble function of the form $\eqref{eq: bubble function}$ and $\psi^{-1}$ be the mapping between unbounded and bounded domains defined in \eqref{eq: mapping for unbounded}. If $\beta\ge r$ with $r\in \mathbb{N}$, then $u_k$ can be extended to $\overline{u}_k: \overline{\Omega} \to \R$ such that $\overline{u}_k$ has bounded mixed derivatives up to the order $r$, i.e.,
\begin{equation*}
    \overline{u}_k \in F_d^r(\overline{\Omega}),\quad \text{ for } k=0,1,\dots,K-1.
\end{equation*}
Furthermore, for $\bm{\alpha}\in \mathbb{N}_0^d$ with $\alpha_j \le r$, there holds
\begin{equation} \label{eq: exact mixed derivatives}
\begin{aligned}
    D^{\bm{\alpha}}\overline{u}_k (\mathbf{z}) =
    \sum_{\substack{\bm{\gamma} + \bm{\zeta} = \bm{\alpha}\\ \bm{\gamma} \preceq \bm{\alpha},\bm{\zeta} \preceq \bm{\alpha} }} \frac{\bm{\alpha}!}{\bm{\gamma}! \bm{\zeta}!} D^{\bm{\gamma}}f_k \prod_{j=1:d, \gamma_j = 0} \left( (1-z_j^2)^{\beta - \zeta_j}Q_{\zeta_j}(z_j)\right) \\
    \times\prod_{j=1:d, \gamma_j\ge 1} \left(\frac{1}{L} (1-z_j^2)^{\beta - \zeta_j - \gamma_j} Q_{\zeta_j}(z_j)P_{\gamma_j}(z_j) \right),
\end{aligned}
\end{equation}
where $Q_{\zeta_j}$ and $P_{\gamma_j}$ are univariate polynomials of degrees $\zeta_j$ and $\gamma_j-1$ defined on $[-1,1]$.
\end{thm}
\begin{proof}
First, for the multi-index $\bm{\alpha}$ with $|\bm{\alpha}| = 0$,
\begin{align*}
 D^{\bm{\alpha}} u_k(\mathbf{z}) = u_k(\mathbf{z}) = F_k(\mathbf{z})b(\mathbf{z}) = f_k(\psi^{-1}(\mathbf{z}))b(\mathbf{z})
 \end{align*}
is continuous and bounded in $\Omega$ due to the continuity and boundedness of $f_k$ in Lemma \ref{lem: 2}. Since $u_k$ approaches zero as $\mathbf{z}\to \partial\Omega$, we can define $\overline{u}_k|_{\partial \Omega} = 0$, and $\overline{u}_k|_{\Omega} = u_k$ for each $k=0,1,\dots, K-1$. Then $\overline{u}_k$ is continuous in the closure $\overline{\Omega}$. The following argument holds for all $k = 0,1,\dots, K-1$. Thus, we drop the subscript $k$ and denote $u:= u_k$, $F:= F_k$ and $f:= f_k$. For any multi-index $\bm{\alpha}\in \mathbb{N}_0^d$, Leibniz's rule implies
\begin{equation} \label{eq: product rule}
    D^{\bm{\alpha}}u = \sum_{\substack{\bm{\gamma} + \bm{\zeta} = \bm{\alpha}\\ \bm{\gamma} \preceq \bm{\alpha},\bm{\zeta} \preceq \bm{\alpha} }} \frac{\bm{\alpha}!}{\bm{\gamma}! \bm{\zeta}!} D^{\bm{\gamma}}F \cdot D^{\bm{\zeta}}b.
\end{equation}
For any fixed $r\in \mathbb{N}$, pick $\beta \ge r$. Consider a multi-index $\bm{\alpha}$ with $\alpha_i\leq r$, $i=1,\ldots,d$.
Let $b(\mathbf{z}) = \prod_{j=1}^d b_j(z_j)$ with $b_j(z_j) = (1-z_j^2)^\beta$. Then direct computation yields
\begin{equation} \label{eq: derivative of b}
    D^{\bm{\zeta}}b = \prod_{j=1}^d \frac{\mathrm{d}^{\zeta_j}b_j}{\mathrm{d} z_j^{\zeta_j}} \quad \mbox{with }   \frac{\mathrm{d}^i b_j}{\mathrm{d} z_j^i} := (1-z_j^2)^{\beta-i} Q_i(z_j)
\end{equation}
where $Q_i(z_j)$, $i=0,1,\ldots,r$, is a polynomial of degree $i$. Note that $F(\mathbf{z})
= f(\psi^{-1}(\mathbf{z}))$ depends only on $z_j$ through $x_j$ with $\mathbf{x} = \psi^{-1}(\mathbf{z})$.
Then we have
\begin{equation} \label{eq: derivative of F}
    D^{\bm{\gamma}}F = D^{\bm{\gamma}}\left(f(\psi^{-1}(\mathbf{z}))\right) = \frac{\partial^{|\bm{\gamma}|}f}{\prod_{j=1}^d \partial x_j^{\gamma_j}} \prod_{j=1:d, \gamma_j\ge 1}\frac{\mathrm{d}^{\gamma_j}x_j}{\mathrm{d} z_j^{\gamma_j}} \quad \mbox{with }   \frac{\mathrm{d}^\ell x_j}{\mathrm{d} z_j^\ell} := \frac{1}{L} (1-z_j^2)^{-\ell} P_\ell(z_j),
\end{equation}
where $P_{\ell}(z_j)$ is a polynomial of degree $\ell-1$ for each $\ell = 1,2,\dots, r$.
Combining the last three identities yields the assertion \eqref{eq: exact mixed derivatives}. Since $\beta - \zeta_j - \gamma_j = \beta - \alpha_j \ge 0$, by Lemma \ref{lem: 2}, we deduce that $D^{\bm{\alpha}}u_k$ is bounded. By defining the continuous extension of $D^{\bm{\alpha}}u_k$ using \eqref{eq: exact mixed derivatives}, we obtain that $\overline{u}_k \in F_d^r(\overline{\Omega})$ for $k=0,1,\dots, K-1$.
\end{proof}

To bound $\overline{u}_k$ in the $F_d^r(\overline{\Omega})$ norm, we need suitable estimates of polynomials $Q_i$ and $P_{\ell}$ in \eqref{eq: exact mixed derivatives}.
\begin{lem} \label{lem: polynomial estimates}
For fixed $\beta\ge r$ with $r\in \mathbb{N}$, $i = 1,2,\dots, r$ and $\ell = 1,2,\dots,r$, the polynomials $Q_i$ defined in \eqref{eq: derivative of b} and $P_{\ell}$ defined in \eqref{eq: derivative of F} satisfy the following estimates
\begin{equation} \label{eq: estimate Q-P}
    \|Q_i\|_{L^\infty([-1,1])} \le \prod_{n=0}^{i-1} \left( n^2 + 2(\beta - n) \right)
\quad \mbox{and}\quad
    \|P_\ell\|_{L^\infty([-1,1])} \le \prod_{n=0}^{\ell-1} (n^2+1).
\end{equation}
\end{lem}
\begin{proof}
By \eqref{eq: derivative of b}, $\|Q_0\|_{L^\infty([-1,1])} = 1$. Using the identity
$\frac{\mathrm{d}}{\mathrm{d} z_j} \left( \frac{\mathrm{d}^{i-1}b_j}{\mathrm{d}z_j^{i-1}} \right) = \frac{\mathrm{d}^i b_j}{\mathrm{d} z_j^i}$, for $i>1$,
and the defining identity in \eqref{eq: derivative of b}, we obtain
\begin{equation*}
\frac{\mathrm{d}}{\mathrm{d} z_j} \left( (1-z_j^2)^{\beta-i+1}Q_{i-1}(z_j) \right) = (1-z_j^2)^{\beta-i}Q_i(z_j).
\end{equation*}
Dividing both sides by $(1-z_j^2)^{\beta-i}$ after taking derivative over the left hand side, we obtain
\begin{equation*}
(\beta-i+1)(-2z_j)Q_{i-1}(z_j) + (1-z_j^2) Q_{i-1}'(z_j) = Q_i(z_j).
\end{equation*}
By the triangular inequality, we derive
\begin{equation}\label{eq:111}
        \|Q_i\|_{L^\infty([-1,1])} \le 2(\beta-i+1) \|Q_{i-1}\|_{L^\infty([-1,1])} + \|Q_{i-1}'\|_{L^\infty([-1,1])}.
    \end{equation}
Since $Q_{i-1}$ is a polynomial of degree $i-1$ defined on $[-1,1]$, a direct application of the Markov brothers' inequality \cite[p. 300]{achieser2013theory} yields
    \begin{equation*}
        \|Q_{i-1}'\|_{L^\infty([-1,1])} \le (i-1)^2 \|Q_{i-1}\|_{L^\infty([-1,1])}.
    \end{equation*}
Plugging this estimate into \eqref{eq:111} leads to the recurrence relation
    \begin{equation} \label{eq: recursion Q}
        \|Q_i\|_{L^\infty([-1,1])} \le \left((i-1)^2 + 2(\beta-(i-1)) \right) \|Q_{i-1}\|_{L^\infty([-1,1])}.
    \end{equation}
Upon noting $\|Q_0\|_{L^\infty([-1,1])} = 1$, we derive the desired estimate on $Q_i$. The estimate on $P_i$ follows similarly.
For $\ell = 1$, $\|P_1\|_{L^\infty([-1,1])} = 1$. For $\ell>1$, using \eqref{eq: derivative of F}, we obtain
\begin{equation*}
\frac{\mathrm{d}}{\mathrm{d} z_j} \left( \frac{\mathrm{d}^{\ell-1}x_j}{\mathrm{d} z_j^{\ell-1}} \right) = \frac{\mathrm{d}}{\mathrm{d} z_j} \left( \frac{1}{L}(1-z_j^2)^{-\ell+1}P_{\ell-1}(z_j) \right) = \frac{1}{L} (1-z_j^2)^{-\ell} P_\ell(z_j) = \frac{\mathrm{d}^\ell x_j}{\mathrm{d} z_j^{\ell}}.
\end{equation*}
This implies
\begin{equation*}
2(\ell-1)z_j P_{\ell-1}(z_j) + (1-z_j^2) P_{\ell-1}'(z_j) = P_\ell(z_j).
\end{equation*}
    Since $P_{\ell-1}$ is a polynomial of degree $\ell-2$, by the triangular inequality and the Markov brothers' inequality, we obtain the recurrence relation
    \begin{equation} \label{eq: recursion P}
        \|P_\ell\|_{L^\infty([-1,1])} \le \left( (\ell-1)^2 + 1 \right) \|P_{\ell-1}\|_{L^\infty([-1,1])}.
    \end{equation}
Together with the identity $\|P_1\|_{L^\infty([-1,1])} = 1$, we prove the desired assertion on $P_\ell$.
\end{proof}

We now estimate the norms $\|D^{\bm{\alpha}}u_k\|_{L^\infty(\overline{\Omega})}$ in terms of suitable mixed norms of $f$. Here,
\begin{align*}
\|f\|_{mix,r} := \max \left\{ \left\|\frac{\partial^{|\bm{\alpha}|} f}{\prod_{j=1}^d \partial x_j^{\alpha_j}}\right\|_{L^\infty(\R^d)} : \mathbf{x}\in \R^d, \bm{\alpha}\in \mathbb{N}_0^d\text{ with } \alpha_j \le r \right\}.
\end{align*}

\begin{thm}[Upper bounds on $\|\overline{u}_k \|_{F_d^1(\overline{\Omega})}$ and $\|\overline{u}_k \|_{F_d^2(\overline{\Omega})}$ for  $k=0,1,\dots,K-1$]\label{thm: small beta}
Let $g$ be the payoff of a put option, and $u_k: \Omega \to \R$ the function defined in \eqref{eq: DP bubble} with $u_k(\mathbf{z}) = f_k(\psi^{-1}(\mathbf{z}))b(\mathbf{z})$. Let $b(\cdot)$ be the bubble function of the form \eqref{eq: bubble function} and $\psi^{-1}$ the mapping between unbounded and bounded domains defined in \eqref{eq: mapping for unbounded}. Then for $k\in \{0,1,\dots,K-1\}$, if $\beta = 1$ in \eqref{eq: bubble function}, there holds
\begin{equation*}
\|\overline{u}_k\|_{F_d^1(\overline{\Omega})} \le \left(2+ \frac{1}{L} \right)^{d}\|f_k\|_{mix,1}.
\end{equation*}
If $\beta = 2$ in \eqref{eq: bubble function}, then
\begin{equation*}
\|\overline{u}_k\|_{F_d^2(\overline{\Omega})} \le \left(24+\frac{12}{L}\right)^d \|f_k\|_{mix,2}.
\end{equation*}
\end{thm}
\begin{proof}
Combining \eqref{eq: exact mixed derivatives} with \eqref{eq: estimate Q-P} leads to
\begin{align} \label{eq: estimate of mixed derivatives}
    \|D^{\bm{\alpha}} \overline{u}_k\|_{L^\infty(\overline{\Omega})} &\le \sum_{\substack{\bm{\gamma} + \bm{\zeta} = \bm{\alpha}\\ \bm{\gamma} \preceq \bm{\alpha},\bm{\zeta} \preceq \bm{\alpha} }} \frac{\bm{\alpha}!}{\bm{\gamma}! \bm{\zeta}!} \|D^{\bm{\gamma}}f_k\|_{L^\infty(\mathbb{R}^d)} \prod_{j=1:d, \gamma_j = 0} \|Q_{\zeta_j}\|_{L^\infty([-1,1])}\\
    &\times\prod_{j=1:d, \gamma_j\ge 1} \left(\frac{1}{L} \|Q_{\zeta_j}\|_{L^\infty([-1,1])} \|P_{\gamma_j}\|_{L^\infty([-1,1])} \right).\nonumber
\end{align}
It follows from Lemma \ref{lem: polynomial estimates} that
\begin{align*}
        \|Q_0\|_{L^\infty([-1,1])} = 1, \quad & \|Q_1\|_{L^\infty([-1,1])} \le 2\beta, \quad \|Q_2\|_{L^\infty([-1,1])} \le (2\beta -1)2\beta,\\  \|P_1\|_{L^\infty([-1,1])} = 1, \quad &\|P_2\|_{L^\infty([-1,1])} \le 2.
\end{align*}
Next, we deduce the following estimates from \eqref{eq: estimate of mixed derivatives}.
If $\beta = 1$, consider the multi-index $\bm{\alpha}$ with $\alpha_j\in \{0,1\}$ for $j=1,2,\dots, d$, then
\begin{equation*}
\begin{aligned}
    \|D^{\bm\alpha}\overline{u}_k\|_{L^\infty(\overline{\Omega})} &\le  \sum_{\substack{\bm{\gamma} + \bm{\zeta} = \bm{\alpha}\\ \bm{\gamma} \preceq \bm{\alpha},\bm{\zeta} \preceq \bm{\alpha} }} \frac{\bm{\alpha}!}{\bm{\gamma}! \bm{\zeta}!} \|f_k\|_{mix,1} \prod_{j=1:d, \gamma_j = 0} 2\beta \prod_{j=1:d, \gamma_j\ge 1} \frac{1}{L} \\
    &=\sum_{\substack{\bm{\gamma} + \bm{\zeta} = \bm{\alpha}\\ \bm{\gamma} \preceq \bm{\alpha},\bm{\zeta} \preceq \bm{\alpha} }}  2^{|\bm{\zeta}|} \left(\frac{1}{L}\right)^{|\bm{\gamma}|} \|f_k\|_{mix,1}  = \left(2+ \frac{1}{L} \right)^{|\bm{\alpha}|}\|f_k\|_{mix,1}\\
    &\leq \left(2+ \frac{1}{L} \right)^{d}\|f_k\|_{mix,1}.
\end{aligned}
\end{equation*}
If $\beta = 2$, consider the multi-index $\bm{\alpha}$ with $\alpha_j\in \{0,1,2\}$ for $j=1,2,\dots,d$,
an application of the trinomial expansion implies
\begin{equation*}
\begin{split}
    \|D^{\alpha}\overline{u}_k\|_{L^\infty(\overline{\Omega})} &\le \|f_k\|_{mix,2} \sum_{\substack{\bm{\gamma} + \bm{\zeta} = \bm{\alpha}\\ \bm{\gamma} \preceq \bm{\alpha},\bm{\zeta} \preceq \bm{\alpha} }} \frac{\bm{\alpha}!}{\bm{\gamma}! \bm{\zeta}!} \prod_{j=1:d, \gamma_j = 0} (2\beta-1)2\beta \prod_{j=1:d, \gamma_j=1} \frac{1}{L}(2\beta) \prod_{j=1:d, \gamma_j=2} \frac{1}{L}\cdot 2 \\
    &\le 2^d \left(2\beta(2\beta-1)+\frac{2\beta}{L}+\frac{2}{L}\right)^d \|f_k\|_{mix,2} =\left(24+\frac{12}{L}\right)^d \|f_k\|_{mix,2}.
\end{split}
\end{equation*}
This completes the proof of the theorem.
\end{proof}

In Theorem \ref{thm: small beta}, we only consider the cases $\beta = 1$ and $\beta=2$, and both can overcome the curse of dimensionality by means of SGPI while maintaining a relatively small upper bound of the functional norm. Clearly, higher order mixed derivatives of the interpolation function $u_k$ have larger upper bounds.

\section{Numerical experiments}\label{sec:numerical}
In this section, we illustrate the efficiency and robustness of the proposed quadrature and sparse grid interpolation scheme, i.e., Algorithm \ref{alg: SGSG}, for pricing high-dimensional American options. We present pricing results up to dimension $16$. The accuracy of the option price $\hat V$ obtained by Algorithm \ref{alg: SGSG} is measured in the relative error defined by
$$ e= |\hat V - V^\dag|/{V^\dag}, $$
where $V^\dag$ is the reference price, either taken from literature or computed to meet a certain tolerance. The results show that the relative errors decay geometrically as the number of interpolation points increase, and the
convergence rate is almost independent of the dimension $d$. The comparison of various
quadrature methods is also included. The implementation of sparse grids is based on the Sparse Grids MATLAB Kit, a MATLAB toolbox for high-dimensional quadrature and interpolation \cite{piazzola.tamellini:SGK}. The computations were performed by MATLAB R2022b with 32 CPU cores (with 4GB memory per core) using research computing facilities offered by Information Technology Services, The University of Hong Kong. The codes for the numerical experiments can be founded in \url{https://github.com/jiefeiy/multi-asset-American-option/tree/main}.

\subsection{American basket option pricing up to dimension 16}
The examples are taken from \cite{kovalov2007pricing}, where
pricing American options up to 6 assets by the FEM was investigated. For each $d$-dimensional problem, the setting of these examples are $S^i_0 = \kappa = 100$, $T = 0.25$, $r = 0.03$, $\delta_i = 0$, $\sigma_i = 0.2$, $P = (\rho_{ij})_{d\times d}$ with $\rho_{ij} = 0.5$ for $i\ne j$. The prices of arithmetic basket options with $d = 2,3,\dots, 12$ underlying assets are listed in Table \ref{tab: SGGK price of arithmBaskPut}, where the last column gives the reference price $V^\dag_{\rm Amer}$ of American options reported in \cite{kovalov2007pricing}, where the relative error is $0.758\%$ for pricing a $6$-d American geometric put option. 

To further illustrate the efficiency of the algorithm in high dimensions, we consider pricing the geometric basket
put options as benchmarks, which can be reduced to a one-dimensional problem. Thus, highly accurate prices
are available using one-dimensional quadrature and interpolation scheme. Indeed, the price
of the $d$-dimensional problem equals that of the one-dimensional American put option with initial price, volatility,
and dividend yield given by
\begin{equation*}
    \hat{S}_0 = \big(\prod_{i=1}^d S_0^i\big)^{1/d}, \quad \hat{\sigma} = \frac{1}{d}\sqrt{\sum_{i,j}\sigma_i \sigma_j \rho_{ij}}, \quad \hat{\delta} = \frac{1}{d}\sum_{i=1}^d \big( \delta_i + \frac{\sigma_i^2}{2}\big) - \frac{\hat{\sigma}^2}{2},
\end{equation*}
respectively. The prices of geometric basket options with $d = 2,3,\dots, 16$ underlying assets are listed
in Table \ref{tab: SGGK price of geoBaskPut}, where the reference Bermudan prices $V^\dag_{\rm Ber}$ with
$50$ times steps with accuracy up to $10^{-6}$ are calculated using one-dimensional quadrature and interpolation scheme. The last column of Table \ref{tab: SGGK price of geoBaskPut} are the reference prices $V^\dag_{\rm Amer}$ of American options reported in \cite{kovalov2007pricing} priced by the reduced one-dimensional problem. 

\begin{table}[hbt!]
\centering
\setlength{\tabcolsep}{5pt}
\begin{threeparttable}
\caption{The prices for the arithmetic basket put option on $d$ assets with $\beta = 1$, $L = 2$, and $K = 50$,
 using sparse grid quadrature with level $4$ Genz-Keister knots. The relative errors in the brackets are compared
 with the American price $V^\dag_{\rm Amer}$.}
\label{tab: SGGK price of arithmBaskPut}\begin{tabular}{cccccccc}
    \toprule
    $d$ & \multicolumn{5}{|c|}{Sparse grid interpolation level $L_I$} & $V_{\rm Amer}^\dag$\\
    \midrule
      & 3 & 4 & 5 & 6 & 7 & \\
    \midrule
    2 & 2.9193 & 3.1397 & 3.1330 & 3.1269 & 3.1388 & 3.13955 \\
      & (7.02e-2) & (3.52e-5) & (2.07e-3) & (4.03e-3) & (2.30e-4) & \\
    \hline
    3 & 2.8649 & 2.9533 & 2.9300 & 2.9304 & 2.9463 & 2.94454 \\
      & (2.71e-2) & (2.96e-3) & (4.95e-3) & (4.81e-3) & (5.83e-4) & \\
    \hline
    4 & 2.8429 & 2.8547 & 2.8232 & 2.8311 &   & 2.84019 \\
      & (9.42e-4) & (5.11e-3) & (5.98e-3) & (3.22e-3) &   &   \\
    \hline
    5 & 2.8271 & 2.7906 & 2.7601 & 2.7710 &   & 2.77193 \\
      & (1.99e-2) & (6.74e-3) & (4.25e-3) & (3.26e-4) &   &   \\
    \hline
    6 & 2.8129 & 2.7455 & 2.7191 & 2.7319 &   & 2.71838 \\
      & (3.48e-2) & (9.96e-3) & (2.60e-4) & (4.97e-3) &   &   \\
    \hline
    7 & 2.7988 & 2.7140 & 2.6913 & 2.7052 &   &  \\
    \hline
    8 & 2.7841 & 2.6908 & 2.6718 &        &   &  \\
    \hline
    9 & 2.7720 & 2.6690 & 2.6553 &        &   &  \\
    \hline
    10 & 2.7599 & 2.6498 & 2.6428 &       &   &   \\
    \hline
    11 & 2.7472 & 2.6331 & 2.6334 &       &   &    \\
    \hline
    12 & 2.7345 & 2.6210 & 2.6256 &       &   &    \\
    \bottomrule
\end{tabular}
\end{threeparttable}
\end{table}

\begin{table}[hbt!]
    \centering
    \setlength{\tabcolsep}{5pt}
    \begin{threeparttable}
    \caption[]{The prices for the geometric basket put option on $d$ assets with $\beta = 1$, $L = 2$, and $K = 50$, using sparse grid quadrature with level $4$ Genz-Keister knots. The relative errors in the brackets are compared with $50$-times exercisable Bermudan price $V^\dag_{\rm Ber}$.}     \label{tab: SGGK price of geoBaskPut}
    \begin{tabular}{cccccccc}
    \toprule
    $d$ & \multicolumn{5}{|c|}{Sparse grid interpolation level $L_I$} & $V^\dag_{\rm Ber}$ & $V^\dag_{\rm Amer}$\\
    \midrule
     & 3 & 4 & 5 & 6 & 7 & &  \\
    \midrule
    2 & 2.9489 & 3.1880 & 3.1880 & 3.1839 & 3.1831 & 3.18310  & 3.18469\\
      & (7.36e-2) & (1.53e-3) & (1.55e-3) & (2.51e-4) & (1.43e-5) &  & \\
    \hline
    3 & 2.9130 & 3.0226 & 3.0060 & 3.0029 & 3.0030 & 3.00299  & 3.00448\\
      & (3.00e-2) & (6.53e-3) & (9.96e-4) & (1.67e-5) & (8.59e-6) &  & \\
    \hline
    4 & 2.9028 & 2.9314 & 2.9103 & 2.9088 &    & 2.90836 & 2.90980\\
      & (1.91e-3) & (7.94e-3) & (6.75e-4) & (1.42e-4) &  &  & \\
    \hline
    5 & 2.8963 & 2.8716 & 2.8514 & 2.8505 &    & 2.84994 & 2.85135\\
      & (1.63e-2) & (7.60e-3) & (5.03e-4) & (2.00e-4) &  &  & \\
    \hline
    6 & 2.8889 & 2.8283 & 2.8110 & 2.8107 &    & 2.81026  & 2.81165\\
      & (2.80e-2) & (6.40e-3) & (2.74e-4) & (1.65e-4) &  &  & \\
    \hline
    7 & 2.8810 & 2.7955 & 2.7817 & 2.7817 &    & 2.78155  & \\
      & (3.57e-2) & (5.03e-3) & (4.95e-5) & (6.93e-5) &  &  & \\
    \hline
    8 & 2.8725 & 2.7718 & 2.7599 &        &    & 2.75980  & \\
      & (4.08e-2) & (4.34e-3) & (3.47e-5) &  &  &  & \\
    \hline
    9 & 2.8631 & 2.7505 & 2.7428 &        &    & 2.74275  & \\
      & (4.39e-2) & (2.84e-3) & (3.37e-5) &  &  &  & \\
    \hline
    10 & 2.8525 & 2.7320 & 2.7292 &       &    & 2.72904  & \\
       & (4.52e-2) & (1.08e-3) & (6.66e-5) &  &  &  & \\
    \hline
    11 & 2.8412 & 2.7154 & 2.7180 &       &    & 2.71776  & \\
       & (4.54e-2) & (8.79e-4) & (8.83e-5) &  &  &  & \\
    \hline
    12 & 2.8294 & 2.7000 & 2.7085 &       &    & 2.70832  & \\
       & (4.47e-2) & (3.09e-3) & (7.74e-5) &  &  &  & \\
    \hline
    13 & 2.8175 & 2.6855 &  &       &    & 2.70031  & \\
       & (4.34e-2) & (5.47e-3) &  &  &  &  & \\
    \hline
    14 & 2.8043 & 2.6720 &  &       &    & 2.69342  & \\
       & (4.12e-2) & (7.95e-3) &  &  &  &  & \\
    \hline
    15 & 2.7910 & 2.6597 &  &       &    & 2.68743  & \\
       & (3.85e-2) & (1.03e-2) &  &  &  &  & \\
    \hline
    16 & 2.7784 & 2.6502 &  &       &    & 2.68218  & \\
       & (3.59e-2) & (1.19e-2) &  &  &  &  & \\
    \bottomrule
    \end{tabular}
    \end{threeparttable}
\end{table}

\subsection{Convergence of interpolation for Bermudan options}
To verify the convergence rate of the SGPI, we consider $50$-times exercisable Bermudan basket put option on the geometric
average of $d$ assets. To avoid the influence of quadrature errors, sparse grid quadrature with level $4$ Genz-Keister knots
is applied to ensure the small approximation errors. Fig. \ref{fig: fig_conv_interp} shows the convergence for different dimension $d$.
These plots show that for a fixed number of inner sparse grids $N$, the relative error do increase with the dimension $d$, but the convergence rate is nearly
independent of the dimension, confirming the theoretical prediction in Section \ref{sec:analysis}.

\begin{figure}[H]
    \centering
    \includegraphics[width=.4\textwidth]{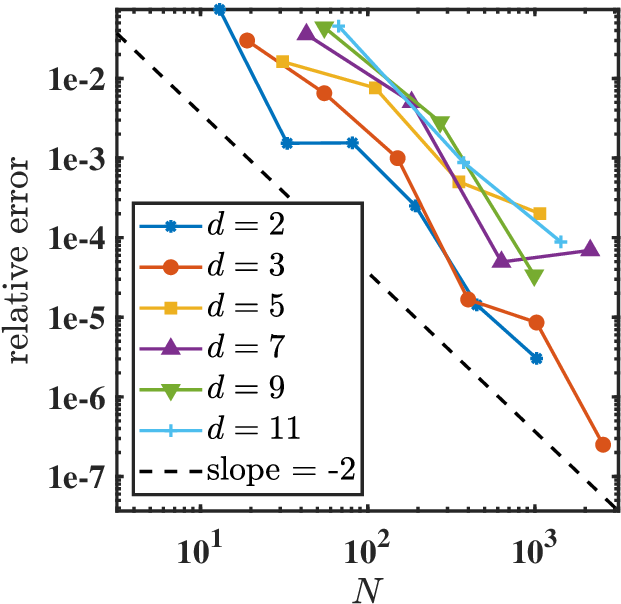}
    \caption{The relative error w.r.t the number of inner sparse grid interpolation points $N$ in dimension $d$ with $\beta = 1$, $L = 2$, and $K = 50$.}
    \label{fig: fig_conv_interp}
\end{figure}

\subsection{Comparison of quadrature}
Now we showcase the performance of Algorithm \ref{alg: SGSG} with different quadrature methods, to demonstrate the flexibility of high-dimensional quadrature rules. Since the pointwise evaluations on the inner sparse grids for interpolation are obtained via quadrature methods, cf. Section \ref{subsec:algorithm}, we present the error of the pricing with three kinds of sparse grid quadrature, the random quasi-Monte Carlo (RQMC) method with scramble Sobol sequence, and the state-of-art preintegration strategy for the integrand with 'kinks' (discontinuity of gradient).

\noindent\textbf{Sparse grid quadrature}: We first show the relative errors of pricing using \eqref{eq: SG quadrature} in Fig. \ref{fig:quad-err}b, for three types of sparse grids, i.e., Gauss-Hermite, Genz-Keister, and normal Leja points for integration with respect to the Gaussian density. The theoretical convergence of the sparse grid quadrature are limited to functions with bounded mixed derivatives, which is not satisfied by $v_{k+1}^{\mathbf{z}}(\cdot)$ defined in \eqref{eq: def of integrand}. Nonetheless, the success of sparse grid quadrature for computing risk-neutral expectations has been observed in the literatures \cite{bungartz2003multivariate, gerstner2007sparse, holtz2010sparse}. Fig. \ref{fig:quad-err}a shows the $L^\infty(\overline{\Omega})$-error of approximating $F_{K-1}$ by sparse grid quadrature, where the exact values correspond to the price of European options with expiration time $\Delta t$. We observe that the quadrature errors in Fig. \ref{fig:quad-err}a seems much larger than the relative errors shown in Fig. \ref{fig:quad-err}b, which seems impausible at the first glance since the latter is poluted by many errors including the former. We find that the quadrature errors are large only near the free interface, which is a $(d-1)$-dimensional manifold in the $d$-dimensional problem, and hence do not result in a heavy impact on the relative errors depicted in Fig. \ref{fig:quad-err}b.

\begin{figure}[H]
    \centering
    \begin{subfigure}[b]{.4\textwidth}
        \centering
        \includegraphics[width = \textwidth]{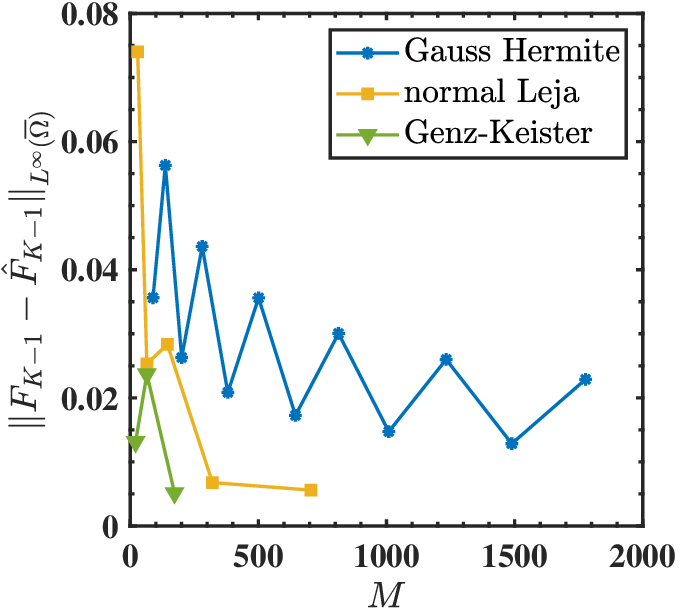}
        \caption{}
        \label{fig: fig_err_quad_a}
    \end{subfigure}
    ~
    \begin{subfigure}[b]{.4\textwidth}
        \centering
        \includegraphics[width = .95\textwidth]{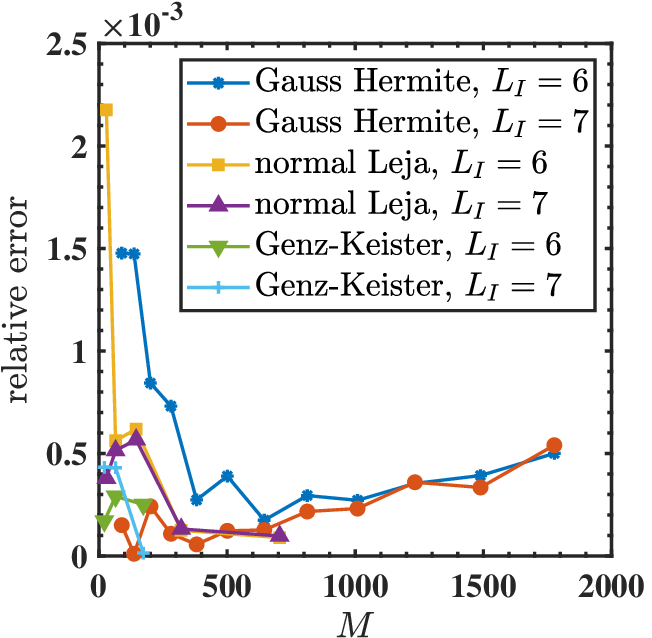}
        \caption{}
        \label{fig: fig_err_quad_b}
    \end{subfigure}  
    \caption{(a) The $L^\infty(\overline{\Omega})$-error w.r.t. the number of quadrature points $M$ for approximating $F_{K-1}$ by sparse grid quadrature rules in dimension $2$. (b) The relative error for pricing the Bermudan geometric basket put option in dimension $2$ with $\beta = 1$, $L = 2$, and $K = 50$.}
    \label{fig:quad-err}
\end{figure}
\begin{figure}[H]
    \centering
    \includegraphics[width=.4\textwidth]{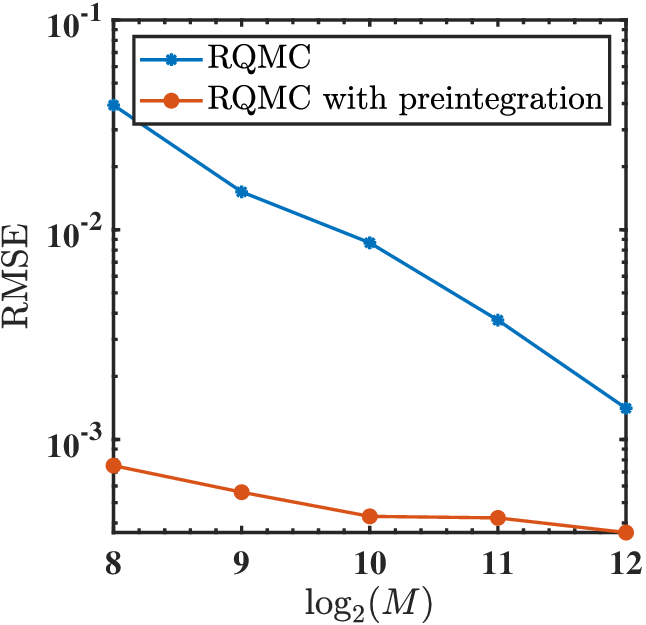}
    \caption{The RMSE w.r.t. the number of scrambled Sobol points $M$ for pricing the 5-d Bermudan geometric basket put option with $L_I = 6$, $\beta = 1$, $L = 2$, and $K = 50$.}
    \label{fig: RMSE preint}
\end{figure}
\noindent\textbf{RQMC and RQMC with preintegration}: To show the convergence with respect to the number of quadrature points $M$, we use randomized quasi-Monte Carlo (RQMC) with the scramble Sobol sequence for quadrature. QMC and RQMC have been widely applied to option pricing problems for computing high-dimensional integrals \cite{giles2008quasi,joy1996quasi,l2009quasi}. The max function in \eqref{eq: def of integrand} introduces a 'kink', which decreases the efficiency of sparse grid quadrature or QMC. For functions with 'kinks', the preintegration strategy or conditional sampling are developed \cite{griewank2018high, liu2023preintegration}. Fig. \ref{fig: RMSE preint} shows the root mean square error (RMSE) of Bermudan option pricing with respect to the number of quadrature points $M$ in a $5$-d problem, where we use $20$ independent replicates to estimate RMSE by
$\text{RMSE} = \sqrt{\frac{1}{20}\sum_{i=1}^{20} ( \hat V^{(i)} - V^\dag )^2}$.

\subsection{Robustness}
To test the robustness of Algorithm \ref{alg: SGSG}, we repeat the experiments for various values of the parameter $\beta$ occurred in the definition of the bubble function \eqref{eq: bubble function}, the scale parameter $L$ introduced in the scaled $\tanh$ map \eqref{eq: mapping for unbounded}, and the number of time steps $K$. The corresponding results are shown in Fig. \ref{fig: fig_stable_beta}, Fig. \ref{fig: fig_stable_L}, and Fig. \ref{fig: fig_stable_K}. We test with the example of pricing Bermudan or American geometric basket put options, where the prices are listed in Table \ref{tab: SGGK price of geoBaskPut}.

Fig. \ref{fig: fig_stable_beta} shows the convergence of relative errors for $\beta = 1,2,3,4,5$. Theoretically we have only provided the upper bounds of $\|\overline{u}_k\|_{F_d^1(\overline{\Omega})}$ and $\|\overline{u}_k\|_{F_d^2(\overline{\Omega})}$ in Theorem \ref{thm: small beta}. The relative errors with respect to the chosen scale parameter $L$ are displayed in Fig. \ref{fig: fig_stable_L}. The smallest relative error is observed for $L = 3.5$. In practice, as mentioned in Section \ref{subsec:unbounded-bounded}, the parameter $L>0$ is determined such that the transformed interpolation points are distributed alike to the asset prices. Theorem \ref{thm: small beta} implies that the scale parameter $L$ should not be too small, and Proposition \ref{prop:bubble} implies that $L$ should not be too large.
\begin{figure}[H]
    \centering
    \begin{subfigure}[b]{0.32\textwidth}
        \centering
        \includegraphics[width=\textwidth]{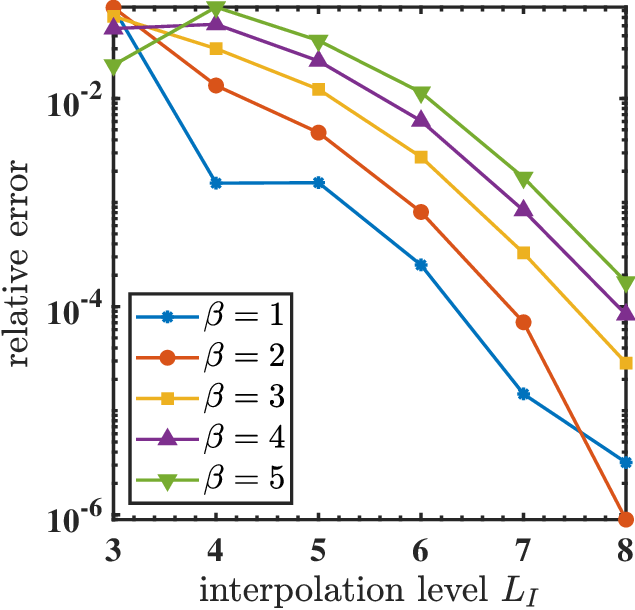}
        \caption{}
        \label{fig: fig_stable_beta}
    \end{subfigure}
    ~
    \begin{subfigure}[b]{0.32\textwidth}
        \centering
        \includegraphics[width=\textwidth]{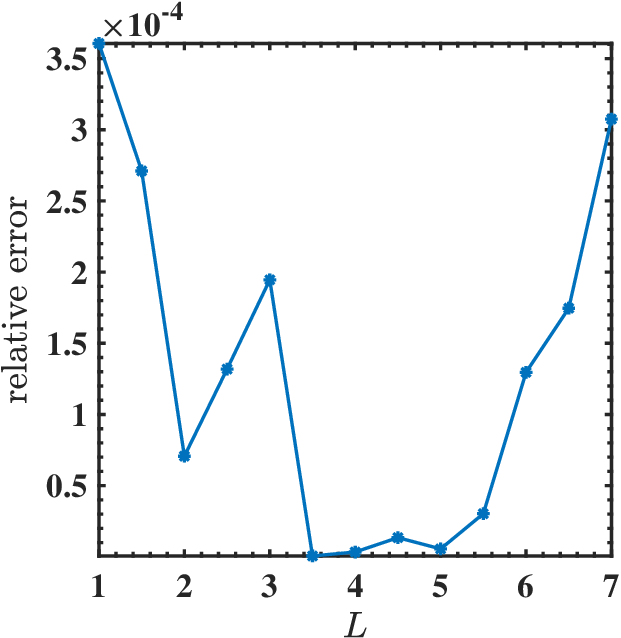}
        \caption{}
        \label{fig: fig_stable_L}
    \end{subfigure}
    ~
    \begin{subfigure}[b]{0.32\textwidth}
        \centering
        \includegraphics[width=\textwidth]{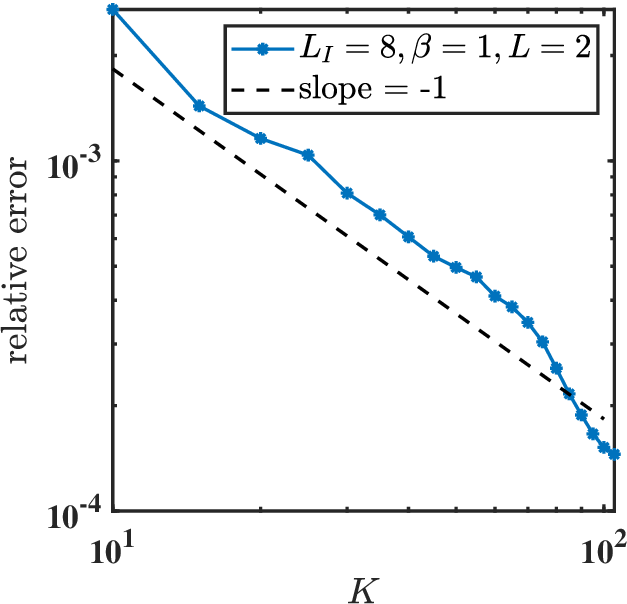}
        \caption{}
        \label{fig: fig_stable_K}
    \end{subfigure}
    \caption{(a) The relative errors decay for various $\beta$ for pricing the Bermudan geometric basket put option in dimension $2$ with $L=2$ and $K=50$. (b) The relative errors w.r.t. the transformation parameter $L$ for pricing the Bermudan geometric basket put option in dimension $2$ with $L_I = 7$, $\beta = 2$, and $K=50$. (c) The convergence of Bermudan prices to American price w.r.t. the number of time steps $K$ for pricing the 2-d Bermudan geometric basket put option with $L_I = 8$, $\beta = 1$, and $L=2$.}
\end{figure}
The time discretization always arises when using the price of $K$-times exercisable Bermudan option to approximate the American price. For the equidistant time step $\Delta t = T/K$, it is widely accepted that the Bermudan price approaches the American price as $\Delta t\to 0$ with a convergence rate $\mathcal{O} (\Delta t)$. For one single underlying asset, this convergence rate was shown in \cite{howison2007matched} for the Black Scholes model. A similar convergence rate has been observed in \cite{quecke2007efficient} and \cite{fang2009pricing} for L\'{e}vy models. In almost all pricing schemes based upon the dynamic programming, a more accurate price can be obtained with more exercise dates (but at a higher computational cost). One general
approach to alleviate the cost but still guarantee the accuracy is to apply the Richardson extrapolation \cite{geske1984american,chang2007richardson}. Fig. \ref{fig: fig_stable_K} present the convergence of Bermudan price to American price as $K$ increases using Algorithm \ref{alg: SGSG}.

\section{Conclusions}\label{sec:conclusion}
In this work, we have developed a novel quadrature and sparse grid interpolation based algorithm for pricing American options with many underlying assets. Unlike most existing methods, it does not involve introducing artificial boundary data by avoiding truncating the computational domain, and that a significant reduction of the number of grid points by introducing a bubble function. The resulting multivariate function has been shown to have bounded mixed derivatives. Numerical experiments for American basket put options with the number of underlying assets up to 16 demonstrate excellent accuracy of the approach. Future work includes pricing max-call options for multiple underlying assets, which are benchmark test cases for high-dimensional American options. Max-call options pose computational challenges due to their unboundedness and thus require certain special treatment.

\bibliographystyle{siam}
\bibliography{references}
\end{document}